\documentclass[11pt,times, 3p, reqno]{amsart}
\usepackage[T1]{fontenc}
\usepackage{a4wide}
\usepackage{hyperref}
\usepackage{enumerate,fullpage}
\usepackage{amssymb,amsmath,graphicx,amsthm}
\usepackage{color}

\newcommand{\be}{\begin{eqnarray*}}
	\newcommand{\en}{\end{eqnarray*}}
\newcommand{\bes}{\begin{eqnarray}}
\newcommand{\ens}{\end{eqnarray}}

\newcommand{\lf}{\left}
\newcommand{\rt}{\right}

\newcommand {\N} {\bf N}

\usepackage{epic, curves}
\def\nn{\nonumber}
\newcommand{\al}{\alpha}

\newcommand{\la}{\lambda}

\newtheorem{theorem}{Theorem}[section]

\newtheorem{definition}{Definition}[section]
\newtheorem{lemma}{Lemma}[section]

\newtheorem{remark}{Remark}[section]

\def\bq{\begin{equation}}
\def\eq{\end{equation}}
\def\bqq{\begin{eqnarray*}}
	\def\eqq{\end{eqnarray*}}

\def\nn{\nonumber}

\title[ 	Convergence rates in expectation  for  a  nonlinear backward parabolic equation  ]{	Convergence rates in expectation  for  a  nonlinear backward parabolic equation  with Gaussian white noise   }

\author[E. Nane]{ Erkan Nane}
\address[E. Nane]{Department of Mathematics and Statistics, Auburn University, Auburn, USA}%
\email{	ezn0001@auburn.edu }

\author[N.H. Tuan]{Nguyen Huy Tuan}
\address[N.H. Tuan]{ Applied Analysis Research Group
	Faculty of Mathematics and Statistics
	Ton Duc Thang University
	Ho Chi Minh City, Viet Nam  }
\email{nguyenhuytuan@tdt.edu.vn}

\begin{document}
	
	\begin{abstract}
		The main purpose of this paper is to study the problem of determining  initial condition of nonlinear parabolic equation from noisy observations of the final condition. We introduce  a  regularized method to establish an approximate solution. We prove an upper  bound on the rate
		of convergence of the mean integrated squared error.
		
	\end{abstract}
	
	\maketitle
	\tableofcontents
 \section{Introduction}
	{
		The forward problem for parabolic equation is of finding the distribution at a later time when we know the initial distribution.  In geophysical exploration, one is often faced
		with the problem of determining the temperature distribution in the object  or any
		part of the Earth at a time $t_0 > 0 $ from temperature measurements at a time $t1 > t_0.$
		This is the backward in time parabolic  problem. The backward parabolic problems can be applied to several practical areas such as image processing, mathematical
		finance, and physics (See \cite{1,2}.) }
Let $T$ be a positive number and $\Omega$ be an open, bounded and connected domain in $\mathbb{R}^d, d \ge 1$ with a smooth boundary $\partial \Omega$.  In this paper, we consider the question of finding the
function $\mathbf u(x,t)$, $(x,t)\in \Omega \times [0,T]$, satisfying the nonlinear problem
\bq
\left\{ \begin{gathered}
	\mathbf u_t-\nabla\Big(a(x,t)\nabla \mathbf u\Big)=F(x,t,\mathbf u(x,t)),\quad (x,t) \in \Omega \times (0,T), \hfill \\
	\mathbf u|_{\partial \Omega}=0,\quad  t \in (0,T),\hfill\\
	\mathbf u(x,T)=g(x),\quad  (x,t) \in \Omega \times (0,T),\hfill  \label{parabolicproblemwhitenoise}
\end{gathered}  \right.
\eq
where the functions $a(x,t), g(x)$ are given and  the source function $F$ will be given later. Here the coefficient $a(x,t)$ is a $C^1$ smooth function and $ 0 < \overline m \le a(x,t) < M $ for  all $(x,t) \in \Omega \times (0,T)$ for some finite constants $\overline m, \ M$.  The problem is well-known to be ill-posed in the sense of Hadamard.   Hence, a solution corresponding to the data
	does not always exist, and in the case of existence, it  does not depend continuously
	on the given data.  In fact, from small noise contaminated physical measurements, the corresponding
	solutions will have large errors.  Hence, one has to resort to a regularization.  In the simple  case of the  deterministic noise, Problem \eqref {parabolicproblemwhitenoise} with $a=1$ and $F=0$ has been studied by many authors { \cite{Denche,Duc,Lesnic}}. However, in the case of random  noise, the analysis  of regularization methods is still limited.
	The problem is to determine the initial temperature function $f$ given a noisy
	version of the temperature distribution $g$ at time $T$
	\begin{equation}
	g_\delta^{\text{obs}}(x)= g(x)+\delta \xi(x)  \label{obs11111}
	\end{equation}
	where
	$\delta >0$
	is the amplitude of the noise	and $\xi$
	is a Gaussian white noise.
	In practice, we only observe some finite errors as follows
	\begin{equation}
	\left \langle	g_\delta , \phi_j \right \rangle = \left \langle	g , \phi_j \right \rangle  +\delta \left \langle	\xi , \phi_j \right \rangle ,\quad j=\overline {1,{\bf N}}=1,2,3,\cdots , {\bf N}.  \label{obs222222}
	\end{equation}
	where the natural number ${\bf N} $  is the number of  steps of  discrete observations {and $\phi_j$ is defined in \eqref{eigenvalue}}.  The main goal is to find approximate solution $\widehat {\bf u}_{N}(0) $ for ${\bf u}(0)$ and then investigate the   rate of convergence 	$ {\bf E}	\| \widehat {\bf u}_{N}(0)   - {\bf u}(0)\|$, which is called the mean integrated square error (MISE). Here $ {\bf E}$ denotes the expectation w.r.t. the distribution of the data in the model \eqref{obs11111}.
		The model \eqref{obs11111}-\eqref{obs222222} are considered in some recent paper, such as \cite{Kekkonen2014,Hohage, Hohage2,Mai, Mair}.

The inverse problem with random noise has a long history.
The
 simple case of \eqref {parabolicproblemwhitenoise}   is the homogeneous linear  parabolic equation of finding the initial data $u_0:=u(x,0)$ that satisfies
\bq
\left\{ \begin{gathered}
\mathbf u_t-\Delta u=0,\quad (x,t) \in \Omega \times (0,T), \hfill \\
\mathbf u|_{\partial \Omega}=0,\quad  t \in (0,T),\hfill\\
\mathbf u(x,T)=g(x),\quad  (x,t) \in \Omega \times (0,T).\hfill  \label{parabolicproblemwhitenoise3}
\end{gathered}  \right.
\eq
This equation  is  a special form of statistical inverse problems and it can be transformed  by a linear operator with random noise
\begin{equation}
g=K u_0 + \text{"noise"}.  \label{K}
\end{equation}
where $K$ is a bounded linear operator that  does not have a continuous inverse.
The  problem \eqref{parabolicproblemwhitenoise3}  has been studied by
 well-known methods  including spectral cut-off (or called truncation method)  \cite{Bi2,Cavalier,Mair,Hohage},  the Tiknonov method \cite{Cox}, iterative regularization methods \cite{Engl}, Bayes estimation method \cite{Bo,Knap}, Lavrentiev regularization method \cite{plato}. In some parts  of these works, the authors show that the error $ {\bf E}	\| \widehat {\bf u}_{N}(0)   - {\bf u}(0)\|$ tend to zero when ${\bf N} $ is suitably chosen according to the value of  $\delta$ and $\delta \to 0$. For more details, we refer the reader  to   \cite{Cavalier1}. \\

To the  best of our  knowledge,  there are no results for  the backward problem for nonlinear parabolic  equation with Gaussian white noise.  The difficulty  to study the nonlinear model is the fact that    we can not transform the  solution of \eqref{parabolicproblemwhitenoise}  into the operator equation  \eqref{K}. This makes the study for nonlinear problem with random noise  more difficult since we can not apply the  known methods.  Very recently, in \cite{kirane-nane-tuan-1}, we studied the discrete random model for backward nonlinear parabolic problem. However, the problem considered in \cite{kirane-nane-tuan-1}  is  in a rectangular  domain which is limited in practice. The present paper uses another random model and also  gives approximation of the solution in the case of  more general  bounded and smooth domain $\Omega$.  Our task in this paper is to show that the  expectation between the  solution and the approximate solution converges  to zero when ${\bf N}$ tends to infinity.

This paper  is organized as follows. In section \ref{section2}, we give a couple of preliminary results. In section \ref{section3}, we give an explanation for ill-posedness of the problem.  For ease of the reader,  we  divide the problem into three cases under various assumptions on  the coefficient $a$, and the source function $F$. \\
{\bf Case 1:} $a:=a(x,t)$ is a constant  and $F$ is a globally Lipschitz function. In section \ref{section4}, we will study this case and give convergence rates in $L^2$ and $H^p$ norms for $p>0$.  The method here is the well-known spectral method. The main idea is  to approximate the final data $g$ by the  approximate data  and use this function to establish a regularized problem by truncation method.

{\bf Case 2:} $a:=a(x,t)$ depends on $x$ and $t$ and $F$ is locally Lipschitz function. This  problem is more difficult. In most practical problems, the function $F$  is often a locally Lipschitz function. The difficulty  here is the fact that   the solution cannot  be transformed into a Fourier series and therefore,  we can not apply  well-known  methods to find an  approximate solution. In Section \ref{section5}, we  will study   a new form  of quasi-reversibility method to construct a regularized solution and obtain convergence rate. Our method is new and very different than the method of Lions and Lattes {\color{red} \cite{lattes}}. First, we approximate the locally Lipschitz function by a sequence of  globally Lipschitz functions and use some new techniques to obtain the convergence rate. \\
 {\bf Case 3} {Various  assumptions on $F$. In practice there are many functions that are not locally Lipschitz. Hence our analysis in section \ref{section4} can not applied in section \ref{section6}.
Our method in  section \ref{section6} is also quasi-reversibility  method and is  very  similar to  the method in section \ref{section4}. But  in section \ref{section6},  we don't approximate $F$ as  we do in  section \ref{section4}. This leads to a convergence rate   that is better than  the one in section \ref{section4}. One difficulty that  occurs in  this section is  showing  the existence and uniqueness of the regularized solution. To prove the existence  of the regularized solution, we don't follow the  previously mentioned  methods. Instead, we use the Faedo -- Galerkin method,  and the compactness method introduced by Lions \cite{Lions}.  To the best of our knowledge, this   is the first result   where    $F$ is not necessarily a locally Lipschitz function. Finally, in section \ref{section7}, we give some specific equations  which can be applied by  our method.
}

\section{Preliminaries}\label{section2}
	To give some details on this random model \eqref{obs11111}, we give the  following definitions	(See \cite{Cavalier1, Cavalier}):
	\begin{definition}
		Let $\mathcal{H}$ be a Hilbert space. Let $g, g_\delta \in \mathcal{H} $  satisfy \eqref{obs11111}. The representation \eqref{obs11111} is equivalent to
		\begin{equation}
		\left \langle	g_\delta , \chi \right \rangle = \left \langle	g , \chi \right \rangle  +\delta  \left \langle	\xi , \chi \right \rangle ,\quad \forall \chi \in \mathcal{H}.  \label{obs2}
		\end{equation}
		Here $\left \langle	\xi , \chi \right \rangle   \sim  \mathcal{N} (0, \|\chi\|^2_\mathcal{H})$. Moreover, given $\chi_1, \chi_2 \in H$ then
		\begin{equation}
		\mathbb{E} \Big( \left \langle	\xi , \chi_1 \right \rangle \left \langle	\xi , \chi_2 \right \rangle \Big)=    \mathbb{E}  \left \langle	\chi_1 , \chi_2 \right \rangle.
		\end{equation}
		
	\end{definition}
	
	\begin{definition}
		The stochastic error is a Hilbert-space process, i.e. a bounded linear operator
		$\xi : \mathcal{H}  \to  L^2(\Omega
		,\mathcal{A}, P)  $ where $(\Omega
		,\mathcal{A}, P)$ is the underlying probability space and $L^2(.,.)$ is the space of all square integrable measurable functions.
	\end{definition}
	
	Let us recall that the eigenvalue problem
	\begin{equation}\quad\quad\begin{cases}
	-\Delta \phi_{j}(x) =\lambda_j \phi_{j}(x), \quad  x \in \Omega, &\\
	\phi_{j}(x)=0, \quad x\in \partial \Omega,&
	\end{cases}
	\end{equation} \label{eigenvalue}
	admits a family of  eigenvalues
	$
	0 < \lambda_1 \le \lambda_2 \le \lambda_3 \le ...\le \lambda_j\le ...\nonumber
	$
	and  $\lambda_j  \to \infty $ as $j  \to \infty $. See page 335 in \cite{evan}.
	
	Next, we introduce the abstract Gevrey class of functions of index $\sigma>0$, see, \emph{e.g.}, \cite{cao}, defined by
	\begin{align*}
	\mathcal{W}_{\sigma}=\Bigg\{v \in L^{2}\left(\Omega\right):\sum_{j=1}^{\infty}e^{2\sigma\lambda_{j}}\left|\big\langle v,\phi_{j}(x)\big\rangle_{L^2(\Omega)} \right|^{2}<\infty\Bigg\} ,
	\end{align*}
	which is a Hilbert space equipped with the inner product
	\begin{align}
	\lf<v_1,v_2\rt>_{\mathcal{W}_{\sigma}}:=\Big< e^{\sigma\sqrt{-\Delta}}v_1,e^{\sigma\sqrt{-\Delta}}v_2 \Big>_{L^2(\Omega)}, \quad \mbox{for all} ~ v_1, v_2 \in \mathcal{W}_{\sigma};\nonumber
	\end{align}
	its corresponding norm
	$
	\left\Vert v\right\Vert _{\mathcal{W}_{\sigma}}=\sqrt{\sum_{j=1}^{\infty}e^{2\sigma\lambda_{j}}\big|\big\langle v,\phi_{j}\big\rangle_{L^2(\Omega)} \big|^{2}}<\infty.
	$
	
\section{The ill-posedness of the nonlinear parabolic equation with random noise}\label{section3}
	
	In this section, for a special case of equation \eqref{parabolicproblemwhitenoise}, we  show that the nonlinear parabolic equation with random noise is ill-posed in the sense of Hadamard.
	\begin{theorem} \label{theorem2.1}
		Problem \eqref{parabolicproblemwhitenoise} is ill-posed in the  special case of $a=1, \Omega=(0, \pi)$.
	\end{theorem}
	\begin{proof}
		Since   $ \Omega =(0,  \pi)$ and $a(x,t)=1$,  Then $\la_{\bf N}={\bf N}^2 $.
		Let us consider the following parabolic equation
		\bq
		\left\{ \begin{gathered}
		\frac{\partial {\bf V}_{\delta, {\bf N(\delta) }} }{\partial t} -\Delta  {\bf V}_{\delta, {\bf N(\delta) }}(t) = F_0({\bf V}_{\delta, {\bf N(\delta) }}(x,t)) ,~~0<t<T, x \in (0, \pi)\hfill \\
		{\bf V}_{\delta, {\bf N(\delta) }}(0,t)= 	{\bf V}_{\delta, {\bf N(\delta) }}(\pi,t)=0,\hfill \\
		{\bf V}_{\delta, {\bf N(\delta) }}(x,T)=   G_{\delta, {\bf N(\delta) }}(x),  \hfill\\
		\end{gathered}  \right. \label{ex1}
		\eq	
		where $F_0$ is
		\begin{equation}
		F_0(v(x,t))= \sum_{j=1}^\infty \frac{e^{- T j^2 } }{2 T}  \left< v(t), \phi_j(x) \right> \phi_j(x)
		\end{equation}
		for any $v \in L^2(\Omega)$,
		and $\phi_j(x)= \sqrt{\frac{2}{\pi}} \sin (jx)$.
		Let us choose 	 $  G_{\delta, {\bf N(\delta) }} \in \mathcal   L^2(\Omega)$ be such that
		\begin{equation}
		  G_{\delta, {\bf N(\delta) }} (x)=  \sum_{j=1}^{\bf N(\delta) }  	\left \langle	g_\delta(x) , \phi_j(x) \right \rangle  \phi_j(x)
		\end{equation}
		where $ g_\delta$
 is defined	by
		\begin{equation}
		\left \langle  g_\delta , \phi_j \right \rangle =\delta  \left \langle	\xi , \phi_j \right \rangle ,\quad j=\overline {1,N}.  \label{obs33333}
		\end{equation}
	 By
	 the usual MISE decomposition which involves a variance term and a
	 bias term, we get
	 \begin{align}
	 {\bf E}	\|  G_{\delta, {\bf N(\delta) }}  \|_{L^2(\Omega)}^2  &={\bf E}	 \Big( \sum_{j=1}^{{\bf N(\delta) }}  	\left \langle	G_{\delta, {\bf N(\delta) }} , \phi_j \right \rangle^2  \Big) =  \delta^2  {\bf E} \Big( \sum_{j=1}^{{\bf N(\delta) }}  \xi_j^2 \Big)= \delta^2 {\bf N(\delta) }.
	 \end{align}	
	The solution of Problem \eqref{ex1} is given by Fourier series (see \cite{Tuan2})
		\begin{align} \label{uexact1}
	{\bf V}_{\delta, {\bf N(\delta) }}(x,t)=  \sum_{ j=1}^\infty \left[e^{(T-t)\lambda_j}  \left \langle	 G_{\delta, {\bf N(\delta) }} , \phi_j \right \rangle - \int_t^T  e^{(s-t) \lambda_j}  \langle	F_0(  {\bf V}_{\delta, {\bf N(\delta) }} (s)	 ) , \phi_j  \rangle  ds \right]\phi_j.
		\end{align}

		We show that Problem \eqref{uexact1} has unique solution $		{\bf V}_{\delta, {\bf N(\delta) }} \in C([0,T]; L^2(\Omega))$. Let us consider
		\begin{align}
		\Phi v:= \sum_{ j=1}^\infty  e^{(T-t)\lambda_j}  \left \langle	 G_{\delta, {\bf N(\delta) }} , \phi_j \right \rangle - \sum_{ j=1}^\infty  \left[\int_t^T  e^{(s-t) \lambda_j}  \langle	F_0( 	v(s)	 ) , \phi_j  \rangle  ds \right] \phi_j.
		\end{align}
		For any $v_1, v_2 \in  C([0,T]; L^2(\Omega)) $, using H\"older inequality, we have for all $t \in [0,T]$
		\begin{align}
		\|	\Phi v_1(t)- \Phi v_2(t)  \|^2_{L^2(\Omega)}&=\sum_{ j=1}^\infty  \left[\int_t^T  e^{(s-t) \lambda_j}  \langle	F_0( 	v_1(s)	 )- F_0( 	v_2(s)	 ) , \phi_j  \rangle  ds \right]^2 \nn\\
		&\le T \sum_{ j=1}^\infty \int_t^T  e^{2(s-t) \lambda_j}  \langle	F_0( 	v_1(s)	 )- F_0( 	v_2(s)	 ) , \phi_j  \rangle^2  ds\nn\\
		&= \frac{T}{4T^2} \sum_{ j=1}^\infty \int_t^T  e^{2(s-t-T) \lambda_j}  \langle	 	v_1(s)	 -  	v_2(s)	  , \phi_j  \rangle^2  ds\nn\\
		&\le \frac{1}{4T} \sum_{ j=1}^\infty \int_t^T    \langle	 	v_1(s)	 -  	v_2(s)	  , \phi_j  \rangle^2   ds  \le \frac{1}{4} \|v_1-v_2\|_{ C([0,T]; L^2(\Omega))}^2.
		\end{align}		
		Hence, we obtain that
		\begin{align}
		\|	\Phi v_1- \Phi v_2  \||_{ C([0,T]; L^2(\Omega))} \le \frac{1}{2} \|v_1-v_2\|_{ C([0,T]; L^2(\Omega))}.
		\end{align}
	This implies that $\Phi$  is a contraction. Using
	the Banach fixed-point theorem, we conclude that the equation $\Phi(w)=w$ has a
	unique solution  $		{\bf V}_{\delta, {\bf N(\delta) }} \in C([0,T]; L^2(\Omega) )$.	
Using the inequality $a^2+b^2 \ge \frac{1}{2} (a-b)^2,~~a, b \in \mathbb{R}$, 	we have the following estimate 
	\begin{align}
		\Big\|  {\bf V}_{\delta, {\bf N(\delta) }} \Big\|_{L^2(\Omega) }^2 & \ge  \underbrace{\frac{1}{2}	\Big\|  \sum_{ j=1}^\infty  e^{(T-t)\lambda_j}  \left \langle	 G_{\delta, {\bf N(\delta) }} ,\phi_j \right \rangle  \phi_j \Big\|_{L^2(\Omega) }^2}_{I_1} \nn\\
		& -  \underbrace {	\Big\|   \sum_{ j=1}^\infty   \left(   \int_t^T  e^{(s-t) \lambda_j}  \langle	F_0( {\bf V}_{\delta, {\bf N(\delta) }}(s)	 ) , \phi_j  \rangle  ds
		\right) \phi_j\Big\|_{L^2(\Omega) }^2  }_{I_2}. \label{es1}
	\end{align}
	First, using H\"older's inequality, we get
	\begin{align}
	I_2 &\le \sum_{ j=1}^\infty  \left(   \int_t^T  e^{(s-t) \lambda_j}  \langle	F_0( 	{\bf V}_{\delta, {\bf N(\delta) }}(s)	 ) , \phi_j  \rangle  ds
	\right)^2 \nn\\
	& \le  T \sum_{ j=1}^\infty  \int_t^T  e^{2(s-t) \lambda_j}  \langle	F_0( 	{\bf V}_{\delta, {\bf N(\delta) }}(s)	 ) , \phi_j  \rangle^2  ds \nn\\
	&\le \frac{T}{4T^2}   \int_t^T \sum_{ j=1}^\infty  e^{2(s-t-T) \lambda_j} \left< 	{\bf V}_{\delta, {\bf N(\delta) }}(t), \phi_j \right>^2ds \le \frac{1}{4} \big\|  	{\bf V}_{\delta, {\bf N(\delta) }} \big\|^2_{C([0,T];L^2(\Omega) )}. \label{es2}
	\end{align}
	And we have the lower bound for $I_1$
	\begin{align}
{\bf  E}	I_1 = \frac{1}{2}  \sum_{ j=1}^\infty e^{2(T-t)\la_j } {\bf  E}  \left \langle	G_{\delta, {\bf N(\delta) }} ,\phi_j \right \rangle^2= \frac{1}{2 } \sum_{j=1}^{\bf N} \delta^2 	e^{2(T-t)\la_j } \ge \frac{1}{2 } \delta^2  e^{2(T-t) \la_{\bf N(\delta) }}. \label{es3}
	\end{align}
	Combining \eqref{es1}, \eqref{es2}, \eqref{es3}, we obtain
	\begin{align}
	{\bf  E}	\Big\|   {\bf V}_{\delta, {\bf N(\delta) }}   \Big\|_{ L^2(\Omega) }^2 +\frac{1}{4} {\bf E} \big\|  {\bf V}_{\delta, {\bf N(\delta) }}  \big\|^2_{C([0,T];L^2(\Omega) )} \ge \frac{1}{2 } \delta^2  e^{2(T-t) \la_{\bf N(\delta) }}.
	\end{align}
	By taking supremum of  both sides on $[0,T]$, we get
	\begin{align}
   {\bf  E} \big\|  {\bf V}_{\delta, {\bf N(\delta) }}  \big\|^2_{C([0,T];L^2(\Omega) )}   \ge \frac{2}{5}  \sup_{0 \le t \le T}\delta^2  e^{2(T-t) \la_{\bf N(\delta) }}	=\frac{2}{5} \delta^2  e^{2T \la_{\bf N(\delta) }}= 	\frac{2}{5} \delta^2  e^{2T {\bf N^2(\delta) }}.
	\end{align}
	Choosing ${\bf N}:={\bf N}(\delta)= \sqrt{ \frac{1}{2T} \ln (\frac{1}{\delta}) }$,  we  obtain
	\begin{align} \label{ob1}
	 {\bf E}	\|  G_{\delta, {\bf N(\delta) }}  \|_{L^2(\Omega)}^2    = \delta^2 {\bf N(\delta) }= \delta^2 	\sqrt{ \frac{1}{2T} \ln (\frac{1}{\delta}) } \to 0,~\text{when}~\delta \to 0.
	\end{align}
		and
		\begin{align} \label{ob2}
	 {\bf E} \big\|  {\bf V}_{\delta, {\bf N(\delta) }}  \big\|^2_{C([0,T];L^2(\Omega) )}   \ge 	 	\frac{2}{5} \delta^2  e^{2T {\bf N^2(\delta) }}	=\frac{2}{5\delta}  \to +\infty,~\text{when}~\delta \to 0.
		\end{align}	
		From \eqref{ob1} and \eqref{ob2}, we can conclude that Problem \eqref{parabolicproblemwhitenoise} is ill-posed.
			\end{proof}

\section{Regularization result  with constant coefficient and globally Lipschitz source  function}\label{section4}
			In this section, we consider the question of finding the
			function $\mathbf u(x,t)$, $(x,t)\in \Omega \times [0,T]$,  that satisfies the problem
			\bq
			\left\{ \begin{gathered}
			\mathbf u_t-\Delta u =F(x,t,\mathbf u(x,t)),\quad (x,t) \in \Omega \times (0,T), \hfill \\
			\mathbf u|_{\partial \Omega}=0,\quad  t \in (0,T),\hfill\\
			\mathbf u(x,T)=g(x),\quad  (x,t) \in \Omega \times (0,T),\hfill  \label{parabolicproblemwhitenoise2}
			\end{gathered}  \right.
			\eq
			Now we have the  following lemma
			\begin{lemma} \label{lemmawhitenoise}
				Let $\overline G_{\delta, {\bf N}(\delta) } \in L^2(\Omega) $  be such that
				\begin{equation}\label{G-delta}
				\overline G_{\delta, {\bf N}(\delta) } = \sum_{j=1}^{\bf N(\delta) } 	\left \langle		g_\delta^{\text{obs}} , \phi_j \right \rangle  \phi_j.
				\end{equation}
				Assume that $ g \in H^{2\gamma}(\Omega) $. Then we have the following estimate
				\begin{equation}
				{\bf E}	\| \overline G_{\delta, {\bf N} (\delta) }  -g \|_{L^2(\Omega)}^2  \le \delta^2 {\bf N}(\delta)  + \frac{1}{\la_{\bf N(\delta)}^{2\gamma}} \|g\|_{H^{2\gamma}(\Omega)}^2
				\end{equation}
				for any $\gamma \ge 0$. Here ${\bf N}$ depends on $\delta$ and satisfies that $\lim_{\delta \to 0} {\bf N}(\delta)  =+\infty$.
			\end{lemma}

			\begin{proof}
				For the
				following proof, we consider the genuine model (\ref{obs222222}). By
				the usual MISE decomposition which involves a variance term and a
				bias term, we get
				\begin{align}
				{\bf E}	\| \overline G_{\delta, {\bf N(\delta) }}  -g \|_{L^2(\Omega)}^2  &={\bf E}	 \Big( \sum_{j=1}^{{\bf N(\delta) }}  	\left \langle		g_\delta^{\text{obs}} -g , \phi_j \right \rangle^2  \Big)+\sum_{j \geq {\bf N(\delta) }+1} 	\left \langle	g , \phi_j \right \rangle^2 \nn\\
				&=  \delta^2  {\bf E} \Big( \sum_{j=1}^{{\bf N(\delta) }}  \xi_j^2 \Big)+ \sum_{j \geq {\bf N(\delta) }+1} \la_j^{-2\gamma}  \la_j^{2\gamma}	\left \langle	g , \phi_j \right \rangle^2
				\end{align}
				Since $\xi_{j}= \langle \xi , \phi_j  \rangle \stackrel {iid}{\sim} N(0,1)$, it follows that  $ {\bf E} \xi_j^2=1$, so
				\begin{align}
				{\bf E}	\| \overline G_{\delta, {\bf N(\delta) }}  -g \|_{L^2(\Omega)}^2   \le  \delta^2 {\bf N(\delta) }+ \frac{1}{\la_{\bf N(\delta) }^{2\gamma}} \|g\|_{H^{2\gamma}}.
				\end{align}

\end{proof}
				Using truncation method, we give a regularized problem for Problem \eqref{parabolicproblemwhitenoise} as follows
				
				\bq
				\left\{ \begin{gathered}
				\frac{\partial }{\partial t}	{\bf u}_{\N(\delta) }^\delta -\Delta 	{\bf u}_{\N(\delta) }^\delta  = {\bf J}_{\al_{\bf N (\delta)}} F(x,t,\mathbf 	{\bf u}_{\N(\delta) }^\delta (x,t)),\quad (x,t) \in \Omega \times (0,T), \hfill \\
				{\bf u}_{\N(\delta) }^\delta  |_{\partial \Omega}=0,\quad  t \in (0,T),\hfill\\
				{\bf u}_{\N(\delta) }^\delta  (x,T)= {\bf J}_{\al_{\bf N (\delta)}} \overline  G_{\delta, {\bf N(\delta) }}(x),\quad  (x,t) \in \Omega \times (0,T),\hfill  \label{whitenoisere}
				\end{gathered}  \right.
				\eq
				where $\al_{\bf N(\delta) }$ is   regularization parameter and $ 	{\bf J}_{\al_{\bf N (\delta)}} $ is the following operator
				\begin{equation}
				{\bf J}_{\al_{\bf N (\delta)}} v:=  \sum_{ \la_j  \le \al_{\bf N (\delta)}   } \Big< v, \phi_j \Big>\phi_j,~~\text{for all}~~v \in  L^2(\Omega).
				\end{equation}
			
			Our  main result in this section  is as follows
			\begin{theorem}\label{thm:regularized-1}
				The problem \eqref{whitenoisere} has a unique solution ${\bf u}^\delta_{\bf N (\delta) } \in C([0,T];L^2(\Omega))$ which satisfies that
				\begin{align} \label{u-ep}
				{\bf u}^\delta_{\bf N(\delta) } (x,t)=  \sum_{ \la_j \le \al_{\bf N (\delta) } } \left[e^{(T-t)\lambda_j}  \left \langle	  \overline G_{\delta, {\bf N(\delta) }} , \phi_j \right \rangle - \int_t^T  e^{(s-t) \lambda_j}  \langle	F( 	{\bf u}^\delta_{\bf N(\delta) } (s)	 ) , \phi_j  \rangle  ds \right]\phi_j.
				\end{align}
				Assume  that problem \eqref{parabolicproblemwhitenoise}  has unique solution ${\bf u}$ such that
				\begin{equation}\label{bound-gevrey-u}
				\sum_{j=1}^\infty   \la_j^{2\beta}  e^{2 t \la_j} \left \langle	{\bf u}(.,t) , \phi_j \right \rangle^2 < A',\quad t \in [0,T].
				\end{equation}
				Let us choose  $\al_{\bf N(\delta) }$ such that
				\begin{equation} \label{cond}
				\lim_{\delta \to 0}  \al_{\bf N(\delta) }= +\infty,~\lim_{ \delta \to 0}  \frac{ e^{kT \al_{\bf N(\delta) }}  }  {\la_{\bf N(\delta) }^{\gamma}   }=0,~~ \lim_{\delta \to 0} e^{kT \al_{\bf N(\delta) }}  \sqrt{{\bf N(\delta)  } } \delta=0
				\end{equation}
				Then the following estimate holds
				\begin{align}
				{\bf E}		\| {\bf u}(.,t)- {\bf u}^\delta_{\bf N(\delta)}(.,t)\|_{L^2(\Omega)}^2  \le  2e^{2k^2 (T-t)} 	e^{-2t\al_{\bf N(\delta)}} \left[ \delta^2 {\bf N(\delta)} 	e^{2T\al_{\bf N(\delta)}}  + \frac{e^{2T\al_{\bf N(\delta)}}}{\la_{\bf N(\delta)}^{2\gamma}} \|g\|_{H^{2\gamma}} +\al_{\bf N(\delta)}^{-2\beta} \right]
				\end{align}
			\end{theorem}
			
			\begin{remark}
				1. From the theorem above, it is easy to see that ${\bf E}	 \lf\|  {\bf u}^\delta_{\bf N(\delta)} (x,t)-\mathbf u(x,t) \rt\|^2_{L^2(\Omega)}$ is of order
				\begin{equation}
				e^{-2t\al_{\bf N(\delta)}}   {\max } \Big( \delta^2 {\bf N(\delta)} 	e^{2T\al_{\bf N(\delta)}}, \frac{e^{2T\al_{\bf N(\delta)}}}{\la_{\bf N(\delta)}^{2\gamma}} ,\al_{\bf N(\delta)}^{-2\beta}    \Big).
				\end{equation}
				
				2. Now, we give one example for the choice of ${\bf N} (\delta)$	 which satisfies the condition \eqref{cond}.  Since $\la_{\bf N} \sim {\bf N}^{\frac{2}{d}} $ { See \cite{Courant}}, we choose $\al_{\bf N}$ such that $e^
				{kT \al_{\bf N(\delta)} }=|{\bf N(\delta)}|^a $ for any $0<a< \frac{2\gamma}{d}$. Then we have $\al_{\bf N(\delta)}  =\frac{a}{kT} \log (\bf N(\delta))$.   The number  ${\bf N(\delta)}$ is chosen as follows
				\[
				{\bf N(\delta)}= \left(  \frac{1}{\delta}\right)^{ba+\frac{b}{2}}
				\]
				for $0<b<1$.  With $ {\bf N} (\delta)$  chosen as  above,   ${\bf E}	 \lf\|  {\bf u}^\delta_{\bf N(\delta)} (x,t)-\mathbf u(x,t) \rt\|^2_{L^2(\Omega)}$ is of order
				$ \left(  \frac{1}{\delta}\right)^{ \frac{-(ba+\frac{b}{2})at} {kT} }$
				
				3. The existence and uniqueness of   solution of equation \eqref{parabolicproblemwhitenoise} is an open problem, and we do not investigate this problem here.  The case considered in Theorem \eqref{theorem2.1} give the existence of the solution of Problem \eqref{parabolicproblemwhitenoise} in a special case.

			\end{remark}

			\begin{proof}[{\bf Proof of Theorem \ref{thm:regularized-1}}]
				We divide the proof into some smaller parts.\\
				{\bf Part 1}. {\it The problem \eqref{whitenoisere} has a unique solution ${\bf u}^\delta_{\bf N(\delta)} \in C([0,T];L^2(\Omega))$.  }
				The proof is similar to \cite{Tuan2}( See Theorem 3.1, page 2975 \cite{Tuan2}).  Hence, we omit it here.\\
				{\bf Part 2}. Estimate the expectation of the error between the exact solution $u$ and the regularized solution ${\bf u}^\delta_{\bf N(\delta)} $.\\
				Let us consider the following integral equation
				\begin{align} \label{v-ep}
				{\bf v}^\delta_{\bf N(\delta)} (x,t)=  \sum_{ \la_j \le \al_{\bf N(\delta)}} \left[e^{(T-t)\lambda_j}  \left \langle	  g , \phi_j \right \rangle - \int_t^T  e^{(s-t) \lambda_j}  \langle	F( 	{\bf v}^\delta_{\bf N (\delta)} (s)	 ) , \phi_j  \rangle  ds \right]\phi_j.
				\end{align}
				We have
				\begin{align}
				\| 	{\bf u}^\delta_{\bf N(\delta)}(.,t)  - 	{\bf v}^\delta_{\bf N(\delta)}(.,t) \|_{L^2(\Omega)}^2 &\le 2    \sum_{ \la_j \le \al_N} e^{2(T-t)\lambda_j}  \left \langle	\overline G_{\delta, {\bf N(\delta)}} - g , \phi_j \right \rangle^2\nn\\
				&+  2    \sum_{ \la_j \le \al_{\bf N (\delta)}} \left[  \int_t^T  e^{(s-t) \lambda_j}   \Big( F_j(	{\bf u}^\delta_{\bf N(\delta)}  )(s)- F_j(	{\bf v}^\delta_{\bf N(\delta)}  )(s)   \Big)   ds \right]^2 \nn\\
				&\le 2 e^{2 (T-t) \al_{\bf N} } \sum_{ \la_j \le \al_{\bf N (\delta)}}  \left \langle	\overline G_{\delta, {\bf N(\delta)}} - g , \phi_j \right \rangle^2\nn\\
				&+2 (T-t) \int_t^T  e^{2(s-t) \al_{\bf N(\delta)} }  \sum_{ \la_j \le \al_{\bf N (\delta)}}  \Big( F_j(	{\bf u}^\delta_{\bf N(\delta)}  )(s)- F_j(	{\bf v}^\delta_{\bf N(\delta)}  )(s)   \Big)^2   ds  \nn\\
				&\le  2 e^{2 (T-t) \al_{\bf N} } 	\| 	\overline G_{\delta, {\bf N(\delta)}} -g \|_{ L^2(\Omega)}^2\nn\\
				&+2k^2 T \int_t^T e^{2(s-t)\al_{\bf N} }  	\| 	{\bf u}^\delta_{\bf N(\delta)}(.,s)  - 	{\bf v}^\delta_{\bf N(\delta)}(.,s) \|_{L^2(\Omega)}^2 	ds.
				\end{align}
				Taking the expectation of both sides of the last inequality, we get
				\begin{align}
				{\bf E}		\| 	{\bf u}^\delta_{\bf N(\delta)}(.,t)  - 	{\bf v}^\delta_{\bf N(\delta)}(.,t) \|_{L^2(\Omega)}^2  &\le 2 e^{2 (T-t) \al_{\bf N(\delta)} } 	{\bf E}	\| \overline G_{\delta, {\bf N(\delta)}} -g \|_{L^2(\Omega)}^2\nn\\
				&+2k^2 T \int_t^T e^{2(s-t)\al_{\bf N} } 		{\bf E}		\| 	{\bf u}^\delta_{\bf N(\delta)}(.,s)  - 	{\bf v}^\delta_{\bf N(\delta)}(.,s) \|_{L^2(\Omega)}^2 ds.
				\end{align}
				Multiplying both sides with $e^{2t\al_{\bf N} }$, we obtain
				\begin{align}
				e^{2t\al_{\bf N(\delta)} } 	{\bf E}		\| 	{\bf u}^\delta_{\bf N(\delta)}(.,t)  - 	{\bf v}^\delta_{\bf N(\delta)}(.,t) \|_{L^2(\Omega)}^2  &\le 2 e^{2 T \al_{\bf N(\delta)} } 	{\bf E}	\| \overline G_{\delta, {\bf N(\delta)}} -g \|_{L^2(\Omega)}^2\nn\\
				&+2k^2 T \int_t^T e^{2s\al_{\bf N(\delta)} } 	{\bf E}		\| 	{\bf u}^\delta_{\bf N(\delta)}(.,s)  - 	{\bf v}^\delta_{\bf N}(.,s) \|_{L^2(\Omega)}^2 ds.
				\end{align}	
				Applying Gronwall's inequality, we get
				\begin{align}
				e^{2t\al_{\bf N(\delta)} } 	{\bf E}		\| 	{\bf u}^\delta_{\bf N (\delta)}(.,t)  - 	{\bf v}^\delta_{\bf N(\delta)}(.,t) \|_{L^2(\Omega)}^2 \le 2 e^{2 T \al_{\bf N(\delta)} }  e^{2k^2 T(T-t)}	{\bf E}	\| \overline G_{\delta, {\bf N(\delta)}} -g \|_{L^2(\Omega)}^2 .
				\end{align}
				Hence, using Lemma \ref{lemmawhitenoise}, we deduce that
				\begin{align}
				{\bf E}		\| 	{\bf u}^\delta_{\bf N(\delta)}(.,t)  - 	{\bf v}^\delta_{\bf N(\delta)}(.,t) \|_{L^2(\Omega)}^2  &\le 2  e^{2k^2 T(T-t)} e^{2 (T-t) \al_{\bf N(\delta)} } 	{\bf E}	\| \overline G_{\delta, {\bf N(\delta)}} -g \|_{L^2(\Omega)}^2  \nn\\
				&\le 2  e^{2k^2 T(T-t)} e^{2 (T-t) \al_{\bf N(\delta)} }  \Big(\delta^2 {\bf N(\delta)}+ \frac{1}{\la_{\bf N(\delta)}^{2\gamma}} \|g\|_{H^{2\gamma}} \Big). \label{saiso1}
				\end{align}
				Now, we continue to estimate $\| {\bf u}(.,t)- {\bf v}^\delta_{\bf N(\delta)}(.,t)\|_{L^2(\Omega)}$. Indeed, using H\"older's inequality and globally Lipschitz property of $F$,  we get
				\begin{eqnarray}
				\begin{aligned}
				&\| {\bf u}(.,t)- {\bf v}^\delta_{\bf N(\delta)}(.,t)\|_{L^2(\Omega)}^2 \nn\\
				& \le 2 	\sum_{ \la_j \le \al_{\bf N (\delta)}} \left[ \int_t^T  e^{(s-t) \lambda_j}   \Big( F_j(	{\bf u}  )(s)- F_j(	{\bf v}^\delta_{\bf N (\delta)}  )(s)   \Big)   ds \right]^2  +2 \sum_{ \la_j > \al_N} \left \langle	{\bf u}(t) , \phi_j \right \rangle^2 \nn\\
				&\le 2 \sum_{ \la_j > \al_N}  \la_j^{-2\beta} e^{-2t \la_j } \la_j^{2\beta} e^{2t \la_j }\left \langle	{\bf u}(t) , \phi_j \right \rangle^2 +2 k^2 \int_t^T e^{2(s-t)\la_N } 	 	\| {\bf u}(.,s)- {\bf v}^\delta_{\bf N(\delta)}(.,s)\|_{L^2(\Omega)}^2 ds \nn\\
				&\le \al_N^{-2\beta} e^{-2t \al_N} 	\sum_{j=1}^\infty   \la_j^{2\beta}  e^{2 t \la_j} \left \langle	{\bf u}(t) , \phi_j \right \rangle^2+2 k^2 \int_t^T e^{2(s-t)\al_{\bf N(\delta)} } 		\| {\bf u}(.,s)- {\bf v}^\delta_{\bf N(\delta)}(.,s)\|_{L^2(\Omega)}^2 ds.
				\end{aligned}
				\end{eqnarray}
				{ Above, we have used  the mild solution of $u$ as follows
						\begin{align*}
						{\bf u} (x,t)=  \sum_{j=1}^\infty \left[e^{(T-t)\lambda_j}  \left \langle	  g , \phi_j \right \rangle - \int_t^T  e^{(s-t) \lambda_j}  \langle	F( {\bf u}(s) ) , \phi_j  \rangle  ds \right]\phi_j.
						\end{align*}
					}
				Multiplying both sides with $e^{2t\al_{\bf N(\delta)} }$, we obtain
				\begin{align}
				e^{2t\al_{\bf N(\delta)}} 	\| {\bf u}(.,t)- {\bf v}^\delta_{\bf N(\delta)}(.,t)\|_{L^2(\Omega)}^2  &\le  \al_{\bf N(\delta)}^{-2\beta} 	\sum_{j=1}^\infty   \la_j^{2\beta}  e^{2 t \la_j} \left \langle	{\bf u}(.,t) , \phi_j \right \rangle^2\nn\\
				&+2 k^2 \int_t^T e^{2s\al_{\bf N} } \| {\bf u}(.,s)- {\bf v}^\delta_{\bf N(\delta)}(.,s)\|_{L^2(\Omega)}^2  ds.
				\end{align}
				Gronwall's inequality implies that
				\begin{equation}
				e^{2t\al_{\bf N}}		\| {\bf u}(.,t)- {\bf v}^\delta_{\bf N(\delta)}(.,t)\|_{L^2(\Omega)}^2  \le e^{2k^2(T-t)}  \al_{\bf N(\delta)}^{-2\beta}  A'.
				\end{equation}
				This together with the estimate \eqref{saiso1} leads to
				\begin{align}
				{\bf E}		\| {\bf u}(.,t)- {\bf u}^\delta_{\bf N(\delta)}(.,t)\|_{L^2(\Omega)}^2 &\le 	2 	{\bf E}		\| 	{\bf u}^\delta_{\bf N}(.,t)  - 	{\bf v}^\delta_{\bf N(\delta)}(.,t) \|_{L^2(\Omega)}^2 + 2 	\| {\bf u}(.,t)- {\bf v}^\delta_{\bf N(\delta)}(.,t)\|_{L^2(\Omega)}^2  \nn\\
				&\le 2e^{2k^2 (T-t) \al_{\bf N}}  \Big(\delta^2 {\bf N(\delta)}+ \frac{1}{\la_{\bf N(\delta)}^{2\gamma}} \|g\|_{H^{2\gamma}} \Big)+2   \al_{\bf N(\delta)}^{-2\beta}	e^{-2t\al_{\bf N}} e^{2k^2(T-t)} A'\nn\\
				\end{align}
where $A'$ is given in equation \eqref{bound-gevrey-u}.
				This completes our proof.
			\end{proof}
			The next theorem provides an error estimate in the Sobolev space $H^p (\Omega)$ which is equipped with a
			norm defined by
			\bes
			\|g\|_{H^p(\Omega)}^2=  \sum\limits_{j=1}^{\infty}    \la_j^p  \Big<g, \phi_{j}(x)\Big>^{2}.
			\ens
			To estimate the error in $H^p$ norm, we need stronger assumption of the  solution $u$.
			\begin{theorem}
				Assume that problem \eqref{parabolicproblemwhitenoise}  has unique solution ${\bf u}$ such that
				\begin{equation} \label{assumption2}
				\sum_{j=1}^\infty    e^{2 (t+r) \la_j} \left \langle	{\bf u}(.,t) , \phi_j \right \rangle^2 < A",\quad t \in [0,T].
				\end{equation}
				for any $r>0$.
				Let us choose  $\al_{\bf N(\delta) }$ such that
				\begin{equation}
				\lim_{\delta \to 0}  \al_{\bf N(\delta) }= +\infty,~\lim_{ \delta \to 0}  \frac{ e^{kT \al_{\bf N(\delta) }}  }  {\la_{\bf N(\delta) }^{\gamma}   }=0,~~ \lim_{\delta \to 0} e^{kT \al_{\bf N(\delta) }}  \sqrt{{\bf N(\delta)  } } \delta=0
				\end{equation}
				Then the following estimate holds
				\begin{align}
				&{\bf E}\|  {\bf u}^\delta_{\bf N(\delta)}(.,t)- {\bf u}(.,t)\|_{H^p(\Omega)}^2 	\\
&\le  2 e^{2k^2 T(T-t)}e^{-2t\al_{\bf N}}   |\al_{\bf N (\delta)}|^p  \left[  2  \delta^2 {\bf N(\delta)} 	e^{2T\al_{\bf N(\delta)}}  + 2\frac{e^{2T\al_{\bf N(\delta)}}}{\la_{\bf N(\delta)}^{2\gamma}} \|g\|_{H^{2\gamma}}+  A" e^{-2r \al_{\bf N_\delta}}  \right] \nn\\
				&+   A" |\al_{\bf N (\delta)}|^p \exp\Big(-2(t+r)  \al_{\bf N(\delta)}  \Big) .
				\end{align}
			\end{theorem}
			\begin{proof}
				First, we have
				\begin{align}
				{\bf E}\|  {\bf u}^\delta_{\bf N(\delta)}(.,t)-{\bf J}_{\al_{\bf N (\delta)}} {\bf u}(.,t)\|_{H^p(\Omega)}^2&= {\bf E} \left( \sum_{ \la_j \le \al_{\bf N (\delta) } } \la_j^{p}  \left \langle  {\bf u}^\delta_{\bf N(\delta)}(x,t)-	{\bf u}(x,t) , \phi_j(x) \right \rangle^2  \right)\nn\\
				& \le |\al_{\bf N (\delta)}|^p {\bf E} \left( \sum_{ \la_j \le \al_{\bf N (\delta) } }   \left \langle  {\bf u}^\delta_{\bf N(\delta)}(x,t)-	{\bf u}(x,t) , \phi_j(x) \right \rangle^2 \right)\nn\\
				&\le |\al_{\bf N (\delta)}|^p  	{\bf E} \|  {\bf u}^\delta_{\bf N(\delta)}(.,t)- {\bf u}(.,t)\|_{L^2(\Omega)}^2. \label{er1}
				\end{align}
				Next, we continue to estimate $	{\bf E} \|  {\bf u}^\delta_{\bf N(\delta)}(.,t)- {\bf u}(.,t)\|_{L^2(\Omega)}^2$ with the assumption \eqref{assumption2}. Let us recall $ {\bf v}^\delta_{\bf N(\delta)}$  from \eqref{v-ep}. The expectation of the error between $ {\bf u}^\delta_{\bf N(\delta)}$ and $ {\bf v}^\delta_{\bf N(\delta)}$ is given in the estimation \eqref{saiso1} as
				\begin{align}
				{\bf E}		\| 	{\bf u}^\delta_{\bf N(\delta)}(.,t)  - 	{\bf v}^\delta_{\bf N(\delta)}(.,t) \|_{L^2(\Omega)}^2  \le  2  e^{2k^2 T(T-t)} e^{2 (T-t) \al_{\bf N(\delta)} }  \Big(\delta^2 {\bf N(\delta)}+ \frac{1}{\la_{\bf N(\delta)}^{2\gamma}} \|g\|_{H^{2\gamma}} \Big). \label{saiso2}
				\end{align}
				Now, we only need to estimate  $\| {\bf u}(.,t)- {\bf v}^\delta_{\bf N(\delta)}(.,t)\|_{L^2(\Omega)}$. Indeed, using H\"older's inequality and globally Lipschitz property of $F$,  we get
				\begin{eqnarray}
				\begin{aligned}
				&\| {\bf u}(.,t)- {\bf v}^\delta_{\bf N(\delta)}(.,t)\|_{L^2(\Omega)}^2 \nn\\
				& \le  2 \sum_{ \la_j > \al_N} \left \langle	{\bf u}(t) , \phi_j \right \rangle^2+ 2 	\sum_{ \la_j \le \al_{\bf N (\delta)}} \left[ \int_t^T  e^{(s-t) \lambda_j}   \Big( F_j(	{\bf u}  )(s)- F_j(	{\bf v}^\delta_{\bf N (\delta)}  )(s)   \Big)   ds \right]^2  \nn\\
				&\le 2 \sum_{ \la_j > \al_N}   e^{-2(t+r) \la_j }  e^{2(t+r) \la_j }\left \langle	{\bf u}(t) , \phi_j \right \rangle^2 +2 k^2T \int_t^T  e^{-2(s-t) \al_{\bf N_\delta}}  	\| {\bf u}(.,s)- {\bf v}^\delta_{\bf N(\delta)}(.,s)\|_{L^2(\Omega)}^2 ds \nn\\
				&\le  e^{-2(t+r) \al_{\bf N_\delta}} 	\sum_{j=1}^\infty   \  e^{2 (t+r) \la_j} \left \langle	{\bf u}(t) , \phi_j \right \rangle^2+2 k^2 T \int_t^T e^{2(s-t)\al_{\bf N(\delta)} } 		\| {\bf u}(.,s)- {\bf v}^\delta_{\bf N(\delta)}(.,s)\|_{L^2(\Omega)}^2 ds.
				\end{aligned}
				\end{eqnarray}
				Multiplying both sides with $e^{2t\al_{\bf N(\delta)} }$, we obtain
				\begin{align}
				e^{2t\al_{\bf N(\delta)}} 	\| {\bf u}(.,t)- {\bf v}^\delta_{\bf N(\delta)}(.,t)\|_{L^2(\Omega)}^2  &\le A"  e^{-2r \al_{\bf N_\delta}}  \nn\\
				&+2 k^2 T \int_t^T 	e^{2s\al_{\bf N(\delta)}} \| {\bf u}(.,s)- {\bf v}^\delta_{\bf N(\delta)}(.,s)\|_{L^2(\Omega)}^2  ds.
				\end{align}
				Gronwall's inequality implies that
				\begin{equation}
				e^{2t\al_{\bf N}}		\| {\bf u}(.,t)- {\bf v}^\delta_{\bf N(\delta)}(.,t)\|_{L^2(\Omega)}^2  \le e^{2k^2T (T-t)}  A"  e^{-2r \al_{\bf N_\delta}} .
				\end{equation}
				This last estimate together with the estimate \eqref{saiso2} leads to
				\begin{align}
				&{\bf E}		\| {\bf u}(.,t)- {\bf u}^\delta_{\bf N(\delta)}(.,t)\|_{L^2(\Omega)}^2 \nn\\
				&\le 	2 	{\bf E}		\| 	{\bf u}^\delta_{\bf N}(.,t)  - 	{\bf v}^\delta_{\bf N(\delta)}(.,t) \|_{L^2(\Omega)}^2 + 2 	\| {\bf u}(.,t)- {\bf v}^\delta_{\bf N(\delta)}(.,t)\|_{L^2(\Omega)}^2  \nn\\
				&\le  4  e^{2k^2 T(T-t)} e^{2 (T-t) \al_{\bf N(\delta)} }  \Big(\delta^2 {\bf N(\delta)}+ \frac{1}{\la_{\bf N(\delta)}^{2\gamma}} \|g\|_{H^{2\gamma}} \Big)+2  e^{2k^2T(T-t)} A"	e^{-2t\al_{\bf N}}  e^{-2r \al_{\bf N_\delta}} \nn\\
				&= 2 e^{2k^2 T(T-t)}e^{-2t\al_{\bf N}}   \left[  2  \delta^2 {\bf N(\delta)} 	e^{2T\al_{\bf N(\delta)}}  + 2\frac{e^{2T\al_{\bf N(\delta)}}}{\la_{\bf N(\delta)}^{2\gamma}} \|g\|_{H^{2\gamma}}+  A" e^{-2r \al_{\bf N_\delta}}  \right]. \label{er2}
				\end{align}
				On the other hand, consider  the function
				\begin{equation}
				G(\xi)= \xi^p e^{-D \xi},~~D>0. \label{ine1}
				\end{equation}
				The derivetive of $G$ is $G'(\xi)=\xi^{p-1} e^{-D\xi} (p-D\xi) $. Hence  we know that
				$G$ is strictly decreasing when $D\xi \ge p$.   Since $\lim_{\delta \to 0} \al_{\bf N(\delta)} =+\infty $, we see that if $\delta$ is small enough then  $2r \al_{\bf N(\delta)}  \ge p$.  Replace $D=2(t+r),~\xi=\al_{\bf N(\delta)}  $ into \eqref{ine1}, we obtain
				for $\la_j > \al_{\bf N(\delta)}  $
				\be
				G(\la_j)=	\la_j^{p}  \exp\Big(-2(t+r) \la_j \Big)  \le G(\al_{\bf N(\delta)} )=  |\al_{\bf N (\delta)}|^p \exp\Big(-2(t+r)  \al_{\bf N(\delta)}  \Big)
				\en
				The latter equatlity leads to
				\begin{align}
				\|  {\bf u}(.,t)-{\bf J}_{\al_{\bf N (\delta)}} {\bf u}(.,t)\|_{H^p(\Omega)}^2&=  \sum_{ \la_j > \al_{\bf N (\delta) } } \la_j^{p}  \left \langle  {\bf u}(x,t) , \phi_j(x) \right \rangle^2 \nn\\
				&=  \sum_{ \la_j > \al_{\bf N (\delta) } } \la_j^{p} \exp\Big(-2(t+r) \la_j \Big) \exp\Big(2(t+r) \la_j \Big)  \left \langle  {\bf u}(x,t) , \phi_j(x) \right \rangle^2\nn\\
				&\le |\al_{\bf N (\delta)}|^p \exp\Big(-2(t+r)  \al_{\bf N(\delta)}  \Big)  \sum_{ \la_j > \al_{\bf N (\delta) } }  \exp\Big(2(t+r) \la_j \Big)  \left \langle  {\bf u}(x,t) , \phi_j(x) \right \rangle^2\nn\\
				&\le A" |\al_{\bf N (\delta)}|^p \exp\Big(-2(t+r)  \al_{\bf N(\delta)}  \Big)  \label{er3}
				\end{align}
				where we use  the assumption \eqref{assumption2} for  the last inequality.
				Combining \eqref{er1}, \eqref{er2} and \eqref{er3}, we deduce that
				\begin{align}
				&{\bf E}\|  {\bf u}^\delta_{\bf N(\delta)}(.,t)- {\bf u}(.,t)\|_{H^p(\Omega)}^2\nn\\
				&\le 	{\bf E}\|  {\bf u}^\delta_{\bf N(\delta)}(.,t)-{\bf J}_{\al_{\bf N (\delta)}} {\bf u}(.,t)\|_{H^p(\Omega)}^2+ 	\|  {\bf u}(.,t)-{\bf J}_{\al_{\bf N (\delta)}} {\bf u}(.,t)\|_{H^p(\Omega)}^2 \nn\\
				&\le  2 e^{2k^2 T(T-t)}e^{-2t\al_{\bf N}}   |\al_{\bf N (\delta)}|^p  \left[  2  \delta^2 {\bf N(\delta)} 	e^{2T\al_{\bf N(\delta)}}  + 2\frac{e^{2T\al_{\bf N(\delta)}}}{\la_{\bf N(\delta)}^{2\gamma}} \|g\|_{H^{2\gamma}}+  A" e^{-2r \al_{\bf N_\delta}}  \right] \nn\\
				&+   A" |\al_{\bf N (\delta)}|^p \exp\Big(-2(t+r)  \al_{\bf N(\delta)}  \Big)
				\end{align}
				which completes the proof.
				
			\end{proof}

	\section{Regularization result  with locally  Lipschitz source }\label{section5}
	Section \ref{section4} has addressed a problem in which $F$ is a global Lipschitz function. In this section we extend the analysis to a locally Lipschitz function $F$. Results for the locally Lipschitz case are  difficult.  Hence, we have to find another regularization method to study the problem with  locally Lipschitz source.\\
	Assume that $a$ is noisy by the observation data $a_\delta^{\text{obs}} :\Omega \times [0,T] \to \mathbb{R}$ as follows
	\begin{equation}
		{\bf a}_\delta^{\text{obs}}(x,t)= a(x,t)+\delta \psi(t)  \label{obs1}
	\end{equation}
	where
	$\delta >0$
	is the amplitude of the noise	and $\psi$
	is Brownian motion in $t$.

	Assume that for each ${\mathcal R}>0$, there exists $K_{\mathcal R}>0$ such that
	\begin{align} \label{local-lip-F}
		|{F} (x,t;u) - {F}(x,t;v)| \leq K_{\mathcal R} |u-v|, ~\mbox{if}~ \max\{|u|, |v|\} \leq {\mathcal R},
	\end{align}
	where $(x,t) \in \Omega \times [0,T]$ and
	\begin{align*}
		K_{\mathcal R}:= \sup \left\{ \left|\frac{{F} (x,t;u) - {F}(x,t;v)}{u-v} \right|: \max\{|u|, |v|\} \leq {\mathcal R}, u \neq v, (x,t) \in \Omega \times [0,T] \right\} < +\infty.
	\end{align*}
	We note that $K_{\mathcal R} $ is increasing and $\lim_{{\mathcal R} \rightarrow +\infty} K_{\mathcal R} = +\infty$. Now, we outline our idea to construct a regularization for the problem \eqref{parabolicproblemwhitenoise}. For all $\mathcal R>0$, we approximate $F$ by $\mathcal{F}_{\mathcal R}$ defined by
	\begin{align} \label{df-F}
		\mathcal{F}_{\mathcal R}(x,t;w) :=
		\begin{cases}
    F(x,t;-\mathcal R), & w \in (-\infty,-\mathcal R)\\
F(x,t;u), & w \in [-\mathcal R, \mathcal R]\\
			F(x,t;\mathcal R), & w \in (\mathcal R,+\infty).
		\end{cases}
	\end{align}
	For each $\delta >0$, we consider a parameter $\mathcal R(\delta) \rightarrow +\infty$ as $\delta \rightarrow 0^+$.
	Let us denote the operator $\mathbb{P} = M \Delta$, where $M$ is a positive number such that $M> 	{\bf a}_\delta^{\text{obs}}(x,t)  $ for all $(x,t) \in \Omega \times (0,T)$.
	Define the following operator
	 $$\mathbb{\bf P}_{\beta_{\bf N(\delta)} }^\delta = \mathbb{P}+\mathbb{\bf Q}_{\beta_{\bf N(\delta)} }^\delta,  $$
	where  \begin{equation}
		\mathbb{\bf Q}_{\beta_{\bf N(\delta)} }^\delta v (x)=\frac{ 1}{T} \sum_{j=1}^\infty   \ln\lf(1+\beta_{\bf N(\delta)} e^{MT \lambda_j}\rt) \big\langle v(x),\phi_j(x)\big\rangle_{L^2(\Omega)} \phi_j(x),
	\end{equation}
	for any function $v \in L^2(\Omega)$.  { Here ${\bf N}(\delta)$ is defined in Lemma \eqref{lemmawhitenoise}. }
	
 Therefore, we are going to introduce the main idea to solve the problem \eqref{parabolicproblemwhitenoise} with a  generalized case of source term defined by \eqref{df-F}, we consider the problem:
	
	\bq  \label{QR2}
	\left\{ \begin{gathered}
		\frac {\partial  {\bf u}^{\delta}_{\bf N(\delta)} }{\partial t}-\nabla\Big(	{\bf a}_\delta^{\text{obs}}(x,t)\nabla {\bf u}^\delta_{\bf N(\delta)}  \Big)-  \mathbb{\bf Q}_{\beta_{\bf N(\delta)} }^\delta  ({\bf u}^\delta_{\bf N(\delta)} ) (x,t)  \hfill \\
		~~~~~~~~~~~~~~~~~~~~~~~~~~~~~~~~~~~~~=\mathcal{F}_{R_\delta} \lf(x,t,{\bf u}^\delta_{\bf N(\delta)}  (x,t)\rt),\quad (x,t) \in \Omega \times (0,T), \hfill \\
		{\bf u}^\delta_{\bf N(\delta)} |_{\partial \Omega }=0,\quad t \in (0,T),\hfill\\
		{\bf u}^\delta_{\bf N}  (x,T)=\overline  G_{\delta, {\bf N(\delta)}}(x), \quad (x,t) \in \Omega \times (0,T),\hfill
	\end{gathered}  \right.
	\eq
Here $\overline  G_{\delta, {\bf N(\delta)}}(x)$ is defined in equation \eqref{G-delta}.
	Now, we introduce some Lemmas which will be  useful for our  main results. First,
	we  recall the abstract Gevrey class of functions of index $\sigma>0$, see, \emph{e.g.}, \cite{cao}, defined by
		\begin{align*}
		\mathcal{W}_{\sigma}=\Bigg\{v \in L^{2}\left(\Omega\right):\sum_{n=1}^{\infty}e^{2\sigma\lambda_{n}}\left|\big\langle v,\phi_{n}(x)\big\rangle_{L^2(\Omega)} \right|^{2}<\infty\Bigg\} ,
		\end{align*}
		which is a Hilbert space equipped with the inner product
		\begin{align}
		\lf<v_1,v_2\rt>_{\mathcal{W}_{\sigma}}:=\Big< e^{\sigma\sqrt{-\Delta}}v_1,e^{\sigma\sqrt{-\Delta}}v_2 \Big>_{L^2(\Omega)}, \quad \mbox{for all} ~ v_1, v_2 \in \mathcal{W}_{\sigma};\nonumber
		\end{align}
		and  corresponding norm
		$
		\left\Vert v\right\Vert _{\mathcal{W}_{\sigma}}=\sqrt{\sum_{n=1}^{\infty}e^{2\sigma\lambda_{n}}\big|\big\langle v,\phi_{n}\big\rangle_{L^2(\Omega)} \big|^{2}}<\infty.
		$

	\begin{lemma}
		\label{lem:F_est} For $\mathcal{F}_{\mathcal R}  \in L^\infty(\Omega \times [0,T] \times \mathbb{R})$, we have
		\begin{align*}
			|\mathcal{F}_{\mathcal R}(x,t;u) - \mathcal{F}_{\mathcal R}(x,t;v)| \leq K_{\mathcal R} |u-v|, ~~~ \forall (x,t)\in \Omega \times [0,T],~ u, v \in \mathbb{R}.
		\end{align*}
	\end{lemma}
	
	\begin{proof}
		See the proof of  Lemma 2.4 in  \cite{Tuan}.
	\end{proof}
	
	\begin{lemma} \label{importantlemma}
		1. Let $M, T>0$. For any $v \in \mathcal{W}_{MT}(\Omega)$, we have
		\begin{equation} \label{bound Q}
			\|  	\mathbb{\bf Q}_{\beta_{\bf N(\delta)} }^\delta ( v)\|_{L^2(\Omega)} \leq  \frac{\beta_{\bf N(\delta)} }{T} \lf\|v\rt\|_{\mathcal{W}_{MT}(\Omega)}.
		\end{equation}
		2. Let  $\beta_{\bf N(\delta)}  < 1- e^{-MT \la_1}$ .  For any $v \in L^2(\Omega)$, we have
		\begin{equation} \label{bound P}
			\left\|	\mathbb{\bf P}_{\beta_{\bf N(\delta)} }^\delta v \right\|_{L^2(\Omega)} \leq \frac{1}{T} \ln\lf(\frac{1}{\beta_{\bf N(\delta)} }\rt)
			\|v\|_{L^2(\Omega)}.
		\end{equation}
	\end{lemma}
\begin{proof}
		Using the inequality $\ln (1+a) \leq a, ~ \forall a >0$, we have
		\begin{align}
		\lf\|	\mathbb{\bf Q}_{\beta_{\bf N(\delta)} }^\delta (v)\rt\|_{L^2(\Omega)}^2 &=\frac{ 1}{T^2} \sum_{j=1}^\infty   \ln^2\lf(1+\beta_{\bf N(\delta)}  e^{MT \lambda_j}\rt) \lf|\big\langle v,\phi_j\big\rangle_{L^2(\Omega)} \rt|^2\nn\\
		&\leq \frac{\beta^2_{\bf N(\delta)} }{T^2} \sum_{j=1}^\infty  e^{2MT \lambda_j}\lf|\big\langle v,\phi_j\big\rangle_{L^2(\Omega)} \rt|^2 \nn\\
		&\leq \frac{\beta^2_{\bf N(\delta)} }{T^2} \lf\|v\rt\|_{\mathcal{W}_{MT}}^2. \label{remark1}
		\end{align}
	Since $\beta_{\bf N(\delta)}  < 1- e^{-MT \la_1}$, we know that $\beta_{\bf N(\delta)} + e^{-MT \lambda_j}<1$ . Using Parseval's equality, we can easily get
		\begin{align}
		\lf\| \mathbb{\bf P}_{\beta_{\bf N(\delta)} }^\delta (v)\rt\|_{L^2(\Omega)}^2 &=\frac{ 1}{T^2} \sum_{j=1}^\infty   \ln^2\lf(\frac{1}{\beta_{\bf N(\delta)} + e^{-MT \lambda_j}}\rt) \lf|\big\langle v,\phi_j\big\rangle_{L^2(\Omega)} \rt|^2\nn\\
		&\leq \frac{1}{T^2} \ln^2\lf(\frac{1}{\beta_{\bf N(\delta)}}\rt) \sum_{j=1}^\infty  \lf|\big\langle v,\phi_j\big\rangle_{L^2(\Omega)} \rt|^2 \nn\\
		&\leq \frac{1}{T^2} \ln^2\lf(\frac{1}{\beta_{\bf N(\delta)}}\rt) \lf\|v\rt\|_{L^2(\Omega)}^2. \nn\label{remark1}
		\end{align}
\end{proof}

	\begin{theorem} \label{nonlocal1}
	The  problem \eqref{QR2} has a unique solution $$  {\bf u}^\delta_{\bf N(\delta)} \in C\left(\left[0,T\right];L^{2}\left(\Omega\right)\right).$$ Assume that the problem \eqref{parabolicproblemwhitenoise} has a unique solution $\mathbf u$ satisfying  $\mathbf u(\cdot,t) \in \mathcal{W}_{MT}$. 		Let us choose $\beta_{\bf N(\delta)}$ such that
	\begin{equation}  \label{dk}
	\lim_{ \delta \to 0}  \delta \sqrt{\bf N(\delta)} \beta_{\bf N(\delta)}^{-1} =	\lim_{ \delta \to 0 } \beta_{\bf N(\delta)}^{-1} \la_{\bf N(\delta)}^{-\gamma}  =	\lim_{ \delta \to 0 } \beta_{\bf N(\delta)}= 0.
	\end{equation}
	 Let us choose $\mathcal{R}_\delta$ such that
	\begin{equation}
	\lim_{\delta \to 0} \beta_{\bf N(\delta)}^{\frac{2t}{T}} e^{ 2K{\mathcal R}_\delta  T}=0,~~t>0. \label{assumption1}
	\end{equation}
	Then we have the following estimate
		\begin{align} \label{error2}
		{\bf E} \lf\|{\bf u}^\delta_{\bf N(\delta)}(x,t)-\textbf{u}(x,t) \rt\|^2_{L^2(\Omega)}
		\leq \beta_{\bf N(\delta)}^{\frac{2t}{T}} e^{( 2K({\mathcal R}_\delta) ) +1 )T} 	\widetilde C (\delta).
		\end{align}
		Here $\widetilde C (\delta)$ is
		\[
		\widetilde C (\delta)=\delta^2 {\bf N}(\delta)  \beta^{-2}_{\bf N_\delta}  + \frac{1}{\la_{\bf N(\delta)}^{2\gamma}  \beta^2_{\bf N_\delta}  } \|g\|_{H^{2\gamma}(\Omega)}
		+ \|\textbf{u}\|_{C\lf([0,T];\mathcal W_{MT}(\Omega)\rt)}^2 +  \frac{\delta^2 T^3 }{b_0 \beta^2_{\bf N_\delta} } \lf\| \textbf{u}  \rt\|_{L^\infty(0,T;\mathcal H_0^1(\Omega))}^2.
		\]
	\end{theorem}
	
		\begin{remark}
			 1. Under the asumption \eqref{assumption1}, the right hand side of equation  \eqref{error2} converges to zero when $t>0$. \\
		 2. Let us choose
		  $\beta_{\bf N(\delta)}= {\bf N(\delta)}^{-c}$ for any $0<c<  \min(\frac{1}{2}, \frac{2\gamma}{d})$. And ${\bf N(\delta)}$ is chosen as follows
			\begin{equation}
			{\bf N(\delta)}= \left(  \frac{1}{\delta} \right)^{m(\frac{1}{2}-c)},~~0<m<1.
			\end{equation}
			Let us choose $\mathcal R_\delta$ such that
		\[
	K\left( 	{\mathcal R}_\delta \right) \le \frac{1}{kT}  \ln \Big( \ln \left( 	{\bf N(\delta)} \right) \Big)=\frac{1}{kT}  \ln \Big(  m(\frac{1}{2}-c  ) \ln \left(  \frac{1}{\delta}	\right) \Big).
		\]
{ Then $	{\bf E} \lf\|{\bf u}^\delta_{\bf N(\delta)}(x,t)-\textbf{u}(x,t) \rt\|^2_{L^2(\Omega)}$ is of order
	\[
	\delta^{ m c(\frac{1}{2}-c)\frac{t}{T} } \ln (\frac{1}{\delta}).
	\]
	}
		\end{remark}

	\begin{proof}[{\bf Proof of Theorem \ref{nonlocal1}}]
	The proof is divided into two Steps. \\
		{\bf Step 1.       The existence and uniqueness of  the solution to the regularized problem \eqref{QR2}.}\\
		Let $b(x,t)$ be defined by $b(x,t)=M-a(x,t)$. It is obvious that $ 0<b(x,t)<M$.
		Then from \eqref{QR2}, we obtain
		\begin{align}
		\frac {\partial  {\bf u}^\delta_{\bf N(\delta)} }{\partial t}+\nabla\Big(b(x,t)\nabla {\bf u}^\delta_{\bf N(\delta)}  \Big)&=F\lf(x,t,{\bf u}^\delta_{\bf N(\delta)}(x,t)\rt)\nn\\ &-\frac{1}{T}\sum_{j=1}^\infty \ln\lf(\frac{1}{\beta_{\bf N(\delta)}+ e^{-MT \lambda_j}}\rt) \big< {\bf u}^\delta_{\bf N(\delta)} (\cdot,t), \phi_j \big> \phi_j(x),
		\end{align}
		for $(x,t) \in \Omega\times (0,T).$\\
		Let ${\bf v}^\delta_{\bf N(\delta)} $ be the function defined by ${\bf v}^\delta_{\bf N(\delta)}(x,t)={\bf u}^\delta_{\bf N(\delta)}(x,T-t).$
		Then we have $$\frac {\partial  {\bf v}^\delta_{\bf N(\delta)} }{\partial t} (x,t) =-\frac {\partial  {\bf u}^\delta_{\bf N(\delta)} }{\partial t}(x,T-t),~\nabla\Big(b(x,t)\nabla {\bf v}^\delta_{\bf N(\delta)} \Big)(x,t) =\nabla\Big(b(x,t)\nabla {\bf u}^\delta_{\bf N(\delta)}\Big)(x,T-t)$$ and
		\begin{align*}
		&\frac{1}{T}\sum_{j=1}^\infty  \ln\lf(\beta_{\bf N(\delta)}+ e^{-MT \lambda_j}\rt) \big< {\bf v}^\delta_{\bf N(\delta)}(x,t), \phi_j(x) \big> \phi_j(x)\nn\\
		&\quad \quad \quad \quad \quad  =\frac{1}{T}\sum_{j=1}^\infty  \ln\lf(\beta_{\bf N(\delta)}+ e^{-MT \lambda_j}\rt) \big< {\bf u}^\delta_{\bf N(\delta)}(x,T-t), \phi_j(x) \big> \phi_j(x).
		\end{align*}
		This implies that  ${\bf v}^\delta_{\bf N(\delta)}$ satisfies the problem
		\bq
		\left\{ \begin{gathered}
		\frac {\partial  {\bf v}^\delta_{\bf N(\delta)} }{\partial t}  -\nabla\Big(b(x,t)\nabla {\bf v}^\delta_{\bf N(\delta)} \Big)=\mathcal{G}(x,t,{\bf v}^\delta_{\bf N} (x,t)) , \quad (x,t) \in \Omega \times (0,T), \hfill \\
		{\bf v}^\delta_{\bf N(\delta)}|_{\partial \Omega}=0, \quad t \in (0,T), \hfill\\
		{\bf v}^\delta_{\bf N(\delta)}(x,0)=\overline  G_{\delta, {\bf N(\delta)}}(x),\quad (x,t) \in \Omega \times (0,T),\hfill  \label{t8}
		\end{gathered}  \right.
		\eq
		where $\mathcal{G}$ is defined by
		\begin{align}
		\mathcal{G}(x,t,v(x,t))&=- F(x,t,v(x,t))\nn\\
		& +\frac{1}{T}\sum_{j=1}^\infty \ln\lf(\frac{1}{\beta_{\bf N(\delta)}+ e^{-MT \lambda_j}}\rt) \big< v(\cdot,t),\phi_j \big>_{L^2(\Omega)} \phi_j(x),
		\end{align}
		for any  $v \in C\lf([0,T];L^2(\Omega)\rt)$. \\
		Since $$\beta_{\bf N(\delta) } \in \lf(0, 1-e^{-MT\la_1}\rt),~~~~~~ 0<\ln\lf(\frac{1}{\beta_{\bf N(\delta)}+ e^{-MT \lambda_n}}\rt) < \ln\lf(\frac{1}{\beta_{\bf N(\delta)}}\rt)$$ and using  Parseval's identity, we obtain for any $v_1, v_2 \in L^2(\Omega)$
		\begin{align}
		&\|\mathcal{G}(\cdot,t,v_1(\cdot,t))-\mathcal{G}(\cdot,t,v_2(\cdot,t)) \|_{L^2(\Omega)} \nn\\
		&\quad \quad \quad \quad \quad \le \|F(\cdot,t,v_1(\cdot,t))-F(\cdot,t,v_2(\cdot,t)) \|_{L^2(\Omega)} \nn\\
		&\quad \quad \quad \quad \quad +\Big\|\frac{1}{T}\sum_{j=1}^\infty \ln\lf(\frac{1}{\beta_{\bf N(\delta)}+ e^{-MT \lambda_j}}\rt) \big< v_1(x,t)-v_2(x,t),\phi_j(x) \big>_{L^2(\Omega)} \phi_j(x) \Big\|_{L^2(\Omega)} \nn\\
		&\quad \quad \quad \quad \quad\le K \|v_1(\cdot,t)-v_2(\cdot,t) \|_{L^2(\Omega)} \nn\\
		&\quad \quad \quad \quad \quad +\frac{1}{T}\sqrt{ \sum_{j=1}^\infty \ln^2\lf(\frac{1}{\beta_{\bf N(\delta)}+ e^{-MT \lambda_j}}\rt) \big| \big< v_1(\cdot,t)-v_2(\cdot,t),\phi_n \big>_{L^2(\Omega)}\big|^2 }\nn\\
		&\quad \quad \quad \quad \quad\le \lf[ K+ \frac{1}{T}\ln\lf(\frac{1}{\beta_{\bf N(\delta)}}\rt) \rt]~\|v_1(\cdot,t)-v_2(\cdot,t) \|_{L^2(\Omega)} .
		\end{align}
		
		So $\mathcal{G}$ is a Lipschitz function. Using the results of Theorem 12.2  in  \cite{chipot}, we complete the proof of Step 1. \\
		
		{\bf Step 2. Error estimate} \\
	 We pass to the error estimate between the regularized solution of problem \eqref{QR2} and the exact solution of problem \eqref{parabolicproblemwhitenoise}. \\
		For $(x,t) \in \Omega \times (0,T)$, we begin by establishing that the functions $b(x,t), 	{\bf b}_\delta^{\text{obs}}(x,t) $ satisfy $$0 < b(x,t) \leq M,~ 0 < b_0\leq 	{\bf b}_\delta^{\text{obs}}(x,t) \leq M$$ and
		\bq \label{subtract}
		{a(x,t) \choose 	{\bf a}_\delta^{\text{obs}}(x,t)} = {M \choose M} - {b(x,t) \choose 	{\bf b}_\delta^{\text{obs}}(x,t)}, ~\forall (x,t) \in \Omega \times (0,T).
		\eq
		The functions $ {\bf u}^\delta_{\bf N(\delta)} (x,t)$ and $\textbf{u}(x,t)$ solve the following equations
		\begin{align}
			\frac{\partial {\bf u}}{\partial t} + \nabla \lf(	{\bf b}_\delta^{\text{obs}}(x,t)\nabla  \textbf{u}\rt)&=F(x,t;{\bf u}(x,t))\nn\\
			& + \nabla \Big(\big(	{\bf b}_\delta^{\text{obs}}(x,t)- b(x,t)\big)\nabla  {\bf u}\Big) +  \mathbb{P}\textbf{u}
			\end{align}
			and
			\begin{align}
			\frac{\partial  {\bf u}^\delta_{\bf N(\delta)} }{\partial t} + \nabla \Big(	{\bf b}_\delta^{\text{obs}}(x,t)\nabla   {\bf u}^\delta_{\bf N(\delta)}  \Big)&= \mathcal{F}_{R_\delta} \lf(x,t,{\bf u}^\delta_{\bf N(\delta)} (x,t)\right)
			+ \mathbb{\bf P}_{\beta_{\bf N(\delta)} }^\delta  {\bf u}^\delta_{\bf N(\delta)}.
		\end{align}
		For $\rho_\delta>0$, we put $$ \mathbf{V}^\delta_{\bf N(\delta)} (x,t)=e^{\rho_\delta (t-T)}\lf[ {\bf u}^\delta_{\bf N(\delta)}(x,t)-\textbf{u}(x,t)\rt].$$ Then for $ (x,t) \in \Omega\times (0,T)$
		\begin{align}\label{eq-v-delta}
			\frac{\partial \mathbf{V}^\delta_{\bf N(\delta)}}{\partial t} &+ \nabla \Big({\bf b}_\delta^{\text{obs}}(x,t)\nabla  \mathbf{V}^\delta_{\bf N(\delta)}\Big)-\rho_\delta \mathbf{V}^\delta_{\bf N(\delta)} \nn\\
			&=
		 \mathbb{\bf P}_{\beta_{\bf N(\delta)} }^\delta   \mathbf{V}^\delta_{\bf N(\delta)}
			+ e^{\rho_\delta (t-T)}  \mathbb{\bf Q}_{\beta_{\bf N(\delta)} }^\delta \textbf{u} - e^{\rho_\delta(t-T)}\nabla \lf(\big(	{\bf b}_\delta^{\text{obs}}(x,t)- b(x,t)\big)\nabla  \textbf{u}\rt) \nn\\
			&\quad + e^{\rho_\delta(t-T)}\lf[\mathcal{F}_{R_\delta} \lf(x,t,{\bf u}^\delta_{\bf N(\delta)} (x,t)\right)- F\big(x,t;\textbf{u}(x,t)\big)\rt],
		\end{align}
		and $$\mathbf{V}^\delta_{\bf N(\delta)}|_{\partial \Omega}=0, ~ \mathbf{V}^\delta_{\bf N(\delta)}(x,T)=\overline  G_{\delta, {\bf N(\delta)}}(x)-g(x) .$$
		By taking the inner product of the two sides of equation \eqref{eq-v-delta} with $\mathbf{V}^\delta_{\bf N(\delta)}$ and noting the  equality $$\int_\Omega \nabla \Big(	{\bf b}_\delta^{\text{obs}}(x,t)\nabla  \mathbf{V}^\delta_{\bf N(\delta)} \Big) \mathbf{V}^\delta_{\bf N(\delta)} dx=-\int_\Omega {\bf b}_\delta^{\text{obs}} (x,t) |\nabla  \mathbf{V}^\delta_{\bf N(\delta)} |^2 dx,$$ we obtain
		\begin{align} \label{3J}
		  &\|\mathbf{V}^\delta_{\bf N(\delta)}(\cdot, T)\|^2_{L^2(\Omega)} - \|\mathbf{V}^\delta_{\bf N(\delta)}(\cdot, t)\|^2_{L^2(\Omega)} \nn\\
		  &~~~~~~~~~~~~~~~~~-2 \int_t^T \int_\Omega  {\bf b}_\delta^{\text{obs}} (x,s) |\nabla  \mathbf{V}^\delta_{\bf N(\delta)} |^2 dxds-2\rho_\delta \int_t^T \|\mathbf{V}^\delta_{\bf N(\delta)}(\cdot, s)\|^2_{L^2(\Omega)}ds \nn\\
		 &~~~~~~~~~~~~~~~~~= \underbrace{ 2\int_t^T  \lf\langle \mathbb{\bf P}_{\beta_{\bf N(\delta)} }^\delta  \mathbf{V}^\delta_{\bf N(\delta)}, \mathbf{V}^\delta_{\bf N(\delta)} \rt\rangle_{L^2(\Omega)} ds}_{=:\widetilde{A_4}} + \underbrace{2 \int_t^T  \lf\langle e^{\rho_\delta(t-T)}  \mathbb{\bf Q}_{\beta_{\bf N(\delta)} }^\delta  \textbf{u}, \mathbf{V}^\delta_{\bf N(\delta)} \rt\rangle_{L^2(\Omega)}ds}_{=:\widetilde{A_5}}\nn\\
		 &~~~~~~~~~~~~~~~~~ + \underbrace{2\int_t^T \lf\langle - e^{\rho_\delta(t-T)}\nabla  \Big(({\bf b}_\delta^{\text{obs}} (x,t)- b(x,t))\nabla  \textbf{u}\Big), \mathbf{V}^\delta_{\bf N(\delta)} \rt\rangle_{L^2(\Omega)}ds}_{=:\widetilde{A_6}} \nn\\
			&~~~~~~~~~~~~~~~~~+\underbrace{2\int_t^T \lf\langle e^{\rho_\delta(t-T)}\lf[\mathcal{F}_{R_\delta} \lf(x,t,{\bf u}^\delta_{\bf N(\delta)} (x,t)\right)- F\big(x,t;\textbf{u}(x,t)\big)\rt], \mathbf{V}^\delta_{\bf N(\delta)} \rt\rangle_{L^2(\Omega)}ds}_{=:\widetilde{A_7}}.
		\end{align}
		First, thanks to inequality \eqref{bound P}, the expectation of  $\widetilde{A_4}$ is estimated as follows
		\begin{align} \label{J1}
		{\bf E}	\big|\widetilde{A_4} \big| &\leq \frac{2}{T} \ln\lf(\frac{1}{\beta_{\bf N_\delta}}\rt)  \int_t^T 	{\bf E}	  \| \mathbf{V}^\delta_{\bf N(\delta)}(\cdot,s) \|_{L^2(\Omega)}^2ds,
		\end{align}
		Next, using  the inequality  \eqref{bound Q} and the H\"older inequality, we have
		\begin{align} \label{J2}
			{\bf E} 	\big|\widetilde{A_5}\big|
			& \leq \int_t^T  e^{2\rho_\beta(s-T)} \frac{\beta_{\bf N_\delta}}{T} \|\textbf{u}\|_{C\lf([0,T];\mathcal{W}_{MT}\rt)}^2 ds   +  \int_t^T 	{\bf E}	 \| \mathbf{V}^\delta_{\bf N(\delta)}(\cdot,s) \|_{L^2(\Omega)}^2ds \nn\\
			&\leq \frac{\beta_{\bf N_\delta}}{T} \|\textbf{u}\|_{C\lf([0,T];\mathcal{W}_{MT}\rt)}^2   + \int_t^T  	{\bf E}	 \| \mathbf{V}^\delta_{\bf N(\delta)}(\cdot,s) \|_{L^2(\Omega)}^2ds,
		\end{align}
	For estimating the expectation of $\big|\widetilde{A_6}\big| $, we use the Green's formula to get the equality
	\[
	\lf\langle  \nabla  \Big(({\bf b}_\delta^{\text{obs}} (x,t)- b(x,t))\nabla  \textbf{u}\Big), \mathbf{V}^\delta_{\bf N(\delta)} \rt\rangle_{L^2(\Omega)} = \lf\langle    \Big(({\bf b}_\delta^{\text{obs}} (x,t)- b(x,t)\Big)  \nabla  \textbf{u}, \nabla \mathbf{V}^\delta_{\bf N(\delta)} \rt\rangle_{L^2(\Omega)}
	\]
	 then using H\"older's inequality and noting the fact that
	\[
	\int_\Omega |\nabla  \textbf{u}(.,s)|^2dx \le  \lf\| \textbf{u} \rt\|_{L^\infty\lf(0,T;\mathcal H_0^1(\Omega)\rt)}^2= \sup_{0 \le s \le T} 	\int_\Omega |\nabla  \textbf{u}(.,s)|^2dx,
	\]
	  we obtain
		\begin{align} \label{J3}
		{\bf E}	\big|\widetilde{A_6}\big|
			&=2 {\bf E}  \lf| \int_t^T \lf\langle e^{\rho_\delta (s-T)}  \Big(({\bf b}_\delta^{\text{obs}} (x,t)- b(x,t)\Big)  \nabla  \textbf{u}, \nabla \mathbf{V}^\delta_{\bf N(\delta)} \rt\rangle_{L^2(\Omega)} ds\rt| \nn\\
			&\leq {\bf E}  \int_t^T  \frac{ e^{2\rho_\delta(s-T)}}{b_0} \int_{\Omega} \Big(({\bf b}_\delta^{\text{obs}} (x,t)- b(x,t)\Big)^2\lf| \nabla \textbf{u} (x,t) \rt|^2 dx ds+ {\bf E}	   \int_t^T  \int_{\Omega} b_0 \lf|\nabla  \mathbf{V}^\delta_{\bf N(\delta)}\rt|^2 dxds\nn\\
			&=\frac{ \delta^2 \int_t^T {\bf E } |\psi(s)|^2 ds \int_\Omega |\nabla  \textbf{u}(.,s)|^2dx  }{b_0}  + {\bf E}	 \int_t^T  \int_{\Omega} b_0 \lf|\nabla  \mathbf{V}^\delta_{\bf N(\delta)}\rt|^2 dxds \nn\\
			&\le \frac{ \delta^2 T^2 }{2b_0} \lf\| \textbf{u} \rt\|_{L^\infty\lf(0,T;\mathcal H_0^1(\Omega)\rt)}^2  +{\bf E}	  \int_t^T  \int_{\Omega} b_0 \lf|\nabla  \mathbf{V}^\delta_{\bf N(\delta)}\rt|^2 dxds.
		\end{align}
		Here in the last inequality, we have used the fact that ${\bf E } |\psi(s)|^2 =s $ since $\psi$ is Brownian motion.
		Finally, since $\lim_{\delta  \rightarrow 0^{+}} {\mathcal R}_\delta=+\infty$, for a sufficiently small $\delta >0$, there is an ${\mathcal R}_\delta>0$ such that $${\mathcal R}_\delta \geq \|\textbf{u}\|_{L^\infty([0,T];L^2(\Omega))}.$$ For this value of ${\mathcal R}_\delta$ we have $$\mathcal{F}_{{\mathcal R}_\delta}\lf(x,t;\textbf{u}(x,t)\rt)=F\lf(x,t;\textbf{u}(x,t)\rt).$$ Using the global Lipschitz property of $
		\mathcal{F}_{{\mathcal R}}$ (see Lemma \ref{lem:F_est}), one obtains  similarly the estimate
		\begin{align} \label{J4}
			{\bf E} \big|\widetilde{A_7}\big|&= 2 {\bf E}\Big| \int_t^T \lf\langle e^{\rho_\delta(t-T)}\lf[\mathcal{F}_{R_\delta} \lf(x,t,{\bf u}^\delta_{\bf N(\delta)} (x,t)\right)- F\big(x,t;\textbf{u}(x,t)\big)\rt], \mathbf{V}^\delta_{\bf N(\delta)} \rt\rangle_{L^2(\Omega)}ds\Big|\nn\\
			& \leq 2 {\bf E} \int_t^T  \lf\| e^{\rho_\delta(t-T)}\lf[\mathcal{F}_{R_\delta} \lf(x,s,{\bf u}^\delta_{\bf N(\delta)} (x,s)\right)- F\big(x,s;\textbf{u}(x,s)\big)\rt]  \rt\|_{L^2(\Omega)}~\|\mathbf{V}^\delta_{\bf N(\delta)}(\cdot,s)\|_{L^2(\Omega)}ds \nn\\
			&\leq  2 K({\mathcal R}_\delta)  \int_t^T  	{\bf E}	 \| \mathbf{V}^\delta_{\bf N(\delta)}(\cdot,s) \|_{L^2(\Omega)}^2ds .
		\end{align}
		Combining  \eqref{3J}, \eqref{J1}, \eqref{J2},\eqref{J3} and \eqref{J4}, we obtain
	
		\begin{align}
		{\bf E} \|\mathbf{V}^\delta_{\bf N(\delta)}(\cdot, T)\|^2_{L^2(\Omega)} &- {\bf E} \|\mathbf{V}^\delta_{\bf N(\delta)}(\cdot, t)\|^2_{L^2(\Omega)}\nn\\
		& + \int_t^T \lf( \frac{\beta_{\bf N_\delta}}{T} \|\textbf{u}\|_{C\lf([0,T];\mathcal{W}_{MT}\rt)}^2   +  \frac{ \delta^2 T^2 }{2b_0} \lf\| \textbf{u} \rt\|_{L^\infty\lf(0,T;\mathcal H_0^1(\Omega)\rt)}^2   \rt) ds \nn\\
			&\geq 2 	{\bf E}  \int_t^T \int_\Omega \left( {\bf b}_\delta^{\text{obs}} (x,s)-b_0\right) |\nabla  \mathbf{V}^\delta_{\bf N(\delta)} |^2 dxds  \nn\\
			&\quad +	{\bf E}  \int_t^T \lf(2\rho_\delta- \frac{2}{T} \ln\lf(\frac{1}{\beta_{\bf N_\delta}}\rt) - 2K({\mathcal R}_\delta) -1 \rt)\|\mathbf{V}^\delta_{\bf N(\delta)} (\cdot,s) \|^2_{L^2(\Omega)} ds\nn\\
			&\geq 	{\bf E}  \int_t^T \lf(2\rho_\delta- \frac{2}{T} \ln\lf(\frac{1}{\beta_{\bf N_\delta}}\rt) - 2K({\mathcal R}_\delta) -1 \rt)\|\mathbf{V}^\delta_{\bf N(\delta)} (\cdot,s) \|^2_{L^2(\Omega)} ds.
		\end{align}
		Whereupon,
		\begin{align}
		{\bf E} \|\mathbf{V}^\delta_{\bf N(\delta)}(\cdot, t)\|^2_{L^2(\Omega)} & \leq {\bf E} \| \overline  G_{\delta, {\bf N(\delta)}}-g \|^2_{L^2(\Omega)}\nn\\
		&+  \beta_{\bf N_\delta}  \|\textbf{u}\|_{C\lf([0,T];\mathcal W_{MT}(\Omega)\rt)}^2 +  \frac{\delta^2 T^3 }{b_0} \lf\| \textbf{u}  \rt\|_{L^\infty(0,T;\mathcal H_0^1(\Omega))}^2 \nn\\
			&\quad + 	{\bf E}  \int_t^T \lf(-2\rho_\delta+ \frac{2}{T} \ln\lf(\frac{1}{\beta_{\bf N_\delta}}\rt) + 2K({\mathcal R}_\delta) +1 \rt)\|\mathbf{V}^\delta_{\bf N(\delta)} (\cdot,s) \|^2_{L^2(\Omega)} ds.
		\end{align}
		Since $$\mathbf{V}^\delta_{\bf N(\delta)}(x,t)=e^{\rho_\delta(t-T)} \Big( {\bf u}^\delta_{\bf N(\delta)}(x,t)-\textbf{u}(x,t)\Big) $$  and applying Lemma \eqref{lemmawhitenoise},  we observe that
		\begin{align}
			&e^{2\rho_\delta(t-T)} {\bf E} \lf\|{\bf u}^\delta_{\bf N(\delta)}(\cdot,t)-\textbf{u}(\cdot,t) \rt\|^2_{L^2(\Omega)}
			\leq \delta^2 {\bf N}(\delta)  + \frac{1}{\la_{\bf N(\delta)}^{2\gamma}} \|g\|_{H^{2\gamma}(\Omega)}   \nn\\
			&~~~~~~~~~~~~~~~~~~~~~~~~~~~~~~~~+ \beta_{\bf N_\delta}  \|\textbf{u}\|_{C\lf([0,T];\mathcal W_{MT}(\Omega)\rt)}^2 +  \frac{\delta^2 T^3 }{b_0} \lf\| \textbf{u}  \rt\|_{L^\infty(0,T;\mathcal H_0^1(\Omega))}^2
			\nn\\
			&~~~~~~~~~~~~~~~~~~~~~~~~~~~~~~~~+ \lf( 2K({\mathcal R}_\delta)  +1 \rt) \int_t^T  	e^{2\rho_\delta(s-T)} {\bf E} \lf\|{\bf u}^\delta_{\bf N(\delta)}(\cdot,s)-\textbf{u}(\cdot,s) \rt\|^2_{L^2(\Omega)}  ds. \label{a}
		\end{align}
		Gronwall's lemma allows us to obtain	\begin{align}
	&	e^{2\rho_\delta(t-T)} {\bf E} \lf\|{\bf u}^\delta_{\bf N(\delta)}(x,t)-\textbf{u}(x,t) \rt\|^2_{L^2(\Omega)} \nn\\
	&		\leq \left[\delta^2 {\bf N}(\delta)  + \frac{1}{\la_{\bf N(\delta)}^{2\gamma}} \|g\|_{H^{2\gamma}(\Omega)}
			+ \beta_{\bf N_\delta}  \|\textbf{u}\|_{C\lf([0,T];\mathcal W_{MT}(\Omega)\rt)}^2 +  \frac{\delta^2 T^3 }{b_0} \lf\| \textbf{u}  \rt\|_{L^\infty(0,T;\mathcal H_0^1(\Omega))}^2  \right] e^{( 2K({\mathcal R}_\delta) ) +1 )(T-t)}.
		\end{align}
		By choosing $\rho_\delta= \frac{1}{T} \ln\lf(\frac{1}{\beta_{\bf N_\delta}}\rt) >0$ we have
		\begin{align} \label{resulproof-a-F_lip}
			 {\bf E} \lf\|{\bf u}^\delta_{\bf N(\delta)}(\cdot,t)-\textbf{u}(\cdot,t) \rt\|^2_{L^2(\Omega)}
			 \leq \beta_{\bf N(\delta)}^{\frac{2t}{T}} e^{( 2K({\mathcal R}_\delta) ) +1 )T} 	\widetilde C (\delta).
		\end{align}
		The proof of Theorem \ref{nonlocal1} is complete.
	\end{proof}
	
	\section{Regularization result  with more general source term}\label{section6}
	In most of the  previous  works on backward nonlinear problem the assumption, that the source is  global or locally Lipschitz, is required. To the best of our knowledge, this section is the first result on the source term $F$ is not  necessarily  a  locally Lipschitz source.
	 We will solve the  problem \eqref{parabolicproblemwhitenoise} with a special generalized case of source term defined by \eqref{df-F}. Our regularized problem is different to  the one in  section \ref{section4} because  we don't approximate the source function $F$. Indeed, we have the following regularized problem
	
	 \bq  \label{QR3}
	 \left\{ \begin{gathered}
	 \frac {\partial  {\bf u}^{\delta}_{\bf N(\delta)} }{\partial t}-\nabla\Big(	{\bf a}_\delta^{\text{obs}}(x,t)\nabla {\bf u}^\delta_{\bf N(\delta)}  \Big)-  \mathbb{\bf Q}_{\beta_{\bf N(\delta)} }^\delta  ({\bf u}^\delta_{\bf N(\delta)} ) (x,t)  \hfill \\
	 ~~~~~~~~~~~~~~~~~~~~~~~~~~~~~~~~~~~~~= {F}  \lf(x,t,{\bf u}^\delta_{\bf N(\delta)}  (x,t)\rt),\quad (x,t) \in \Omega \times (0,T), \hfill \\
	 {\bf u}^\delta_{\bf N(\delta)} |_{\partial \Omega }=0,\quad t \in (0,T),\hfill\\
	 {\bf u}^\delta_{\bf N}  (x,T)=\overline  G_{\delta, {\bf N(\delta)}}(x), \quad (x,t) \in \Omega \times (0,T),\hfill
	 \end{gathered}  \right.
	 \eq

We make the following assumptions on $F \in C^0 (\mathbb R)$ in the following:  There exists $C_1$ and $C'_1, C_2$ and $p >1$ and $\overline   \gamma$ such that
\begin{align}
&zF(x,t,z) \ge C_1 |z|^p-C'_1  \label{ass1} \\
&|F(x,t,z)| \le C_2(1+|z|^{p-1}) \label{ass2} \\
&(z_1-z_2) \left( F(x,t,z_1)-F(x,t,z_2) \right) \ge -\overline \gamma |z_1-z_2|^2. \label{ass3}
\end{align}
It  is easy to check that the function {$F(x,t, z)= z^{\frac{1}{3}}$ }satisfies the conditions \eqref{ass1}, \eqref{ass2} and \eqref{ass3}. Note here  that this function is not locally Lipschitz.

Now we have the following result

		\begin{theorem} \label{nonlocal2}
			Let us assume that $F$ satisfies \eqref{ass1}, \eqref{ass2} and \eqref{ass3}. Then, there exists a unique  weak solution  $ {\bf u}^\delta_{\bf N(\delta)} $ of  problem \eqref{QR3} such that
			 $$  {\bf u}^\delta_{\bf N(\delta)} \in L^{2}(0,T;H^{1})\cap L^{\infty }(0,T;L^{2}).$$ Assume that the problem \eqref{parabolicproblemwhitenoise} has a unique solution $\mathbf u$ satisfying  $\mathbf u(\cdot,t) \in \mathcal{W}_{MT}$. 	Let us choose $\beta_{\bf N_\delta}$ be as Theorem \eqref{nonlocal1}.
			Then we have the following estimate
			\begin{align} \label{error3}
			{\bf E} \lf\|{\bf u}^\delta_{\bf N(\delta)}(x,t)-\textbf{u}(x,t) \rt\|^2_{L^2(\Omega)}
			\leq \beta_{\bf N(\delta)}^{\frac{2t}{T}} e^{( 2\overline \gamma +1 )T} 	\widetilde C (\delta).
			\end{align}
{ where $\widetilde C (\delta)$ is defined in \eqref{mainerror}. }
\end{theorem}

\begin{remark}
	Our method in this Theorem give the convergence rate \eqref{error3}  which better than the error rate in  \eqref{error2}. Indeed, since $\lim_{\delta \to 0}  K({\mathcal R}_\delta) =+\infty $, we have
	\begin{equation}
	\frac{\text{The right hand side of \eqref{error2}}}{\text{The right hand side of \eqref{error3}}}   =\frac{\beta_{\bf N(\delta)}^{\frac{2t}{T}} e^{( 2K({\mathcal R}_\delta) ) +1 )T} 	\widetilde C (\delta)}{  \beta_{\bf N(\delta)}^{\frac{2t}{T}} e^{( 2\overline \gamma +1 )T} 	\widetilde C (\delta)} \to +\infty
	\end{equation}
	when $\delta \to 0$.
\end{remark}

\subsection{ Proof of Theorem \ref{nonlocal2}}

\subsubsection{Proof of the existence  of solution of Problem \eqref{QR3} }
First, by changing variable ${\bf v}^\delta_{\bf N(\delta)} (x,t)= {\bf u}^\delta_{\bf N(\delta)}(x,T-t)$, we transform Problem \eqref{QR3} into the initial value problem
	\bq
	\left\{ \begin{gathered}
	\frac {\partial  {\bf v}^\delta_{\bf N(\delta)} }{\partial t}  -\nabla\Big({\bf b}_\delta^{\text{obs}}(x,t)\nabla {\bf v}^\delta_{\bf N(\delta)} \Big)=-F(x,t,{\bf v}^\delta_{\bf N} (x,t)) +\mathbb{\bf P}_{\beta_{\bf N(\delta)} }^\delta  ( {\bf v}^\delta_{{\bf N(\delta)}}  (x,t) )  , \quad (x,t) \in \Omega \times (0,T), \hfill \\
	{\bf v}^\delta_{\bf N(\delta)}|_{\partial \Omega}=0, \quad t \in (0,T), \hfill\\
	{\bf v}^\delta_{\bf N(\delta)}(x,0)=\overline  G_{\delta, {\bf N(\delta)}}(x),\quad (x,t) \in \Omega \times (0,T).\hfill  \label{QR4}
	\end{gathered}  \right.
	\eq
where ${\bf b}_\delta^{\text{obs}}(x,t)= M- {\bf a}_\delta^{\text{obs}}(x,t)$.\\
The weak formulation of the initial boundary value  problem  \eqref{QR4}
can then be given in the following manner: Find ${\bf v}^\delta_{{\bf N(\delta)}}(t)$ defined in the open
set $(0,T)$ such that ${\bf v}^\delta_{{\bf N(\delta)}}$ satisfies the following variational problem
\begin{align}
\int_\Omega  \frac{d}{dt} {\bf v}^\delta_{{\bf N(\delta)},m}  \varphi  dx &+ \int_\Omega {\bf b}_\delta^{\text{obs}}(x,t)\nabla {\bf v}^\delta_{{\bf N(\delta)},m} \nabla \varphi  dx+ 	\int_\Omega F(  {\bf v}^\delta_{{\bf N(\delta)},m}   (t)  ) \varphi  dx \nn\\
&= \int_\Omega \mathbb{\bf P}_{\beta_{\bf N(\delta)} }^\delta  ( {\bf v}^\delta_{{\bf N(\delta)},m}  (t) )  \varphi dx  \label{pt11}
\end{align}
for all $ \varphi \in H^{1},$ and the initial condition%
\begin{equation}
\begin{tabular}{l}
$ {\bf v}^\delta_{\bf N(\delta)} (0)=\overline  G_{\delta, {\bf N(\delta)}}.$%
\end{tabular}
\end{equation}

	\textbf{Proof of the existence  of solution of Problem \eqref{QR3} }. The proof consists of several steps.

	\textbf{Step 1:} \textit{The Faedo -- Galerkin approximation} (introduced by
	Lions \cite{Lions}).\\
	 In the space $H^1(\Omega)$, we take a basis
	 $ \{e_{j} \}_{j=1}^\infty$ and define the finite dimensional subspace
	 \[
	 V_m= {\text span} \{ e_1, e_2,...e_m \}.
	 \]
Let $\overline  G_{\delta, {\bf N(\delta)},m} $ be an element of $V_m$ such that	
	\begin{equation} \label{pt00}
	\begin{tabular}{l}
	$\overline  G_{\delta, {\bf N(\delta)},m} =\sum_{j=1}^{m}  d^\delta_{mj} e_{j}\rightarrow \overline  G_{\delta, {\bf N(\delta)}} $ strongly in $%
	L^{2} $%
	\end{tabular}
	\end{equation}
	as $m \to +\infty$.
 We can  express  the approximate
	solution of the problem  \eqref{QR4} in the form%
	\begin{equation}
	 {\bf v}^\delta_{\bf N(\delta),m}(t)=\sum_{j=1}^{m}c^\delta_{mj}(t) e_{j}, \label{pt000}
	\end{equation}%
	where the coefficients $c^\delta_{mj}$ satisfy the system of linear differential
	equations
	\begin{align}
	\int_\Omega  \frac{d}{dt} {\bf v}^\delta_{{\bf N(\delta)},m}  e_i dx &+ \int_\Omega {\bf b}_\delta^{\text{obs}}(x,t)\nabla {\bf v}^\delta_{{\bf N(\delta)},m} \nabla e_i dx+ 	\int_\Omega F(  {\bf v}^\delta_{{\bf N(\delta)},m}   (t)  ) e_i dx \nn\\
	&= \int_\Omega \mathbb{\bf P}_{\beta_{\bf N(\delta)} }^\delta  ( {\bf v}^\delta_{{\bf N(\delta)},m}  (t) )  e_i dx  \label{pt1}
	\end{align}
	with  $i=\overline {1,m}$ and  the initial conditions
	\begin{equation}
	c^\delta_{mj} (0)= d^\delta_{mj},~j=\overline {1,m}. \label{pt2}
	\end{equation}
	The existence of a local solution of system \eqref{pt1}-\eqref{pt2} is guaranteed by Peano's theorem on existence of solutions. For each $m$ there exists a solution ${\bf v}^\delta_{\bf N(\delta),m}(t)$ in the  form \eqref{pt000} which satisfies \eqref{pt1} and \eqref{pt2}  almost everywhere on $%
	0\leq t\leq T_{m}$ for some $T_{m},$ $0<T_{m}\leq T.$ The following
	estimates allow one to take $T_{m}=T$\ for all $m.$

	\textbf{Step 2. A priori estimates}.

	a) \textit{\bf The first estimate}. Multiplying the $i^{th}$\ equation of (\ref%
	{pt1})\ by $c^\delta_{mi}(t)$\ and summing up with respect to $i,$ afterwards,
	integrating by parts with respect to the time variable from $0$ to $t,$ we
	get after some rearrangements
	\begin{align}
	&\left\Vert {\bf v}^\delta_{\bf N(\delta),m}(t) \right\Vert_{L^2(\Omega)}
	^{2}+2 \int_0^t  \int_\Omega {\bf b}_\delta^{\text{obs}}(x,t)|\nabla {\bf v}^\delta_{{\bf N(\delta)},m}(s) |^2  dx ds  +2\int\nolimits_{0}^{t} \int_\Omega
	F(  {\bf v}^\delta_{{\bf N(\delta)},m}   (s)  )   {\bf v}^\delta_{{\bf N(\delta)},m}   (s)  dx ds \nn\\
	&=\left\Vert \overline  G_{\delta, {\bf N(\delta)},m}\right\Vert
	^{2} + 2 \int\nolimits_{0}^{t} \int_\Omega
	\mathbb{\bf P}_{\beta_{\bf N(\delta)} }^\delta  (  {\bf v}^\delta_{{\bf N(\delta)},m}   (s)  )   {\bf v}^\delta_{{\bf N(\delta)},m}   (s)  dx ds
	\end{align}
	By \eqref{pt00},  we have%
	\begin{equation}
	\begin{tabular}{l}
	$\left\Vert \overline  G_{\delta, {\bf N(\delta)},m}\right\Vert
	^{2} \leq B_{0}(\delta),$ \ for all \ $m,$%
	\end{tabular}
	\tag{3.8}  \label{c8}
	\end{equation}%
	where $B_{0}(\delta)$ depends  on $\overline  G_{\delta, {\bf N(\delta)}}$ and is  independent of $m$. \\
		Using the lower bound of $ {\bf b}_\delta^{\text{obs}}(x,t)$, we have the following estimate
	\begin{equation}
	2 \int_0^t  \int_\Omega {\bf b}_\delta^{\text{obs}}(x,t)|\nabla {\bf v}^\delta_{{\bf N(\delta)},m}(s) |^2  dx ds \ge 2b_0 \int_0^t \|{\bf v}^\delta_{{\bf N(\delta)},m}(s) \|_{H^1(\Omega)}ds. \label{pt4}
	\end{equation}
	Using the assumption on $F$, we have
	\begin{align}
	2\int\nolimits_{0}^{t} \int_\Omega
	F(  {\bf v}^\delta_{{\bf N(\delta)},m}   (s)  )   {\bf v}^\delta_{{\bf N(\delta)},m}   (s)  dx ds  \ge 	2C_{1}\int\nolimits_{0}^{t}\left\Vert  {\bf v}^\delta_{{\bf N(\delta)},m}   (s) \right\Vert
	_{L^{p}(\Omega)}^{p}ds-2TC_{1}^{\prime } \label{pt5}
	\end{align}
	and
	\begin{align}
	2 \int\nolimits_{0}^{t} \int_\Omega
	\mathbb{\bf P}_{\beta_{\bf N(\delta)} }^\delta  (  {\bf v}^\delta_{{\bf N(\delta)},m}   (s)  )   {\bf v}^\delta_{{\bf N(\delta)},m}   (s)  dx ds  \le \frac{2}{T} \ln\lf(\frac{1}{\beta_{\bf N(\delta)} }\rt) \int\nolimits_{0}^{t} \left\Vert  {\bf v}^\delta_{{\bf N(\delta)},m}   (s)  \right\Vert_{L^2(\Omega)}^{2}ds. \label{pt6}
	\end{align}
	 Hence, it follows from (\ref{pt4}) -- (\ref{pt6}) that
	 \begin{align}
	 \left\Vert {\bf v}^\delta_{\bf N(\delta),m}(t) \right\Vert_{L^2(\Omega)}
	 ^{2} &+ 2b_0 \int_0^t \|{\bf v}^\delta_{{\bf N(\delta)},m}(s) \|_{H^1(\Omega)}ds + 2C_{1}\int\nolimits_{0}^{t}\left\Vert  {\bf v}^\delta_{{\bf N(\delta)},m}   (s) \right\Vert
	 _{L^{p}(\Omega)}^{p}ds  \nn\\
	 & \le B_{0}(\delta)+2TC_{1}^{\prime }+ \frac{1}{T} \ln\lf(\frac{1}{\beta_{\bf N(\delta)} }\rt) \int\nolimits_{0}^{t} \left\Vert {\bf v}^\delta_{{\bf N(\delta)},m}(s) \right\Vert_{L^2(\Omega)}^{2}ds. \label{pt8}
	 \end{align}
	Let us denote
	\begin{equation}
	{\bf S}_m^\delta (t)= \left\Vert {\bf v}^\delta_{\bf N(\delta),m}(t) \right\Vert_{L^2(\Omega)}
	^{2} + 2b_0 \int_0^t \|{\bf v}^\delta_{{\bf N(\delta)},m}(s) \|_{H^1(\Omega)}ds + 2C_{1}\int\nolimits_{0}^{t}\left\Vert  {\bf v}^\delta_{{\bf N(\delta)},m}   (s) \right\Vert
	_{L^{p}(\Omega)}^{p}ds. \label{eq1}
	\end{equation}
	 Using the fact that $ \int\nolimits_{0}^{t} \left\Vert  {\bf v}^\delta_{{\bf N(\delta)},m}   (s)  \right\Vert_{L^2(\Omega)}^{2}ds \le   \int\nolimits_{0}^{t} 	{\bf S}_m^\delta (s) ds $, we know from \eqref{pt8} that
	 \begin{align}
	 	{\bf S}_m^\delta (t) \le B_{0}(\delta)+2TC_{1}^{\prime }+ \frac{1}{T} \ln\lf(\frac{1}{\beta_{\bf N(\delta)} }\rt)  \int\nolimits_{0}^{t} 	{\bf S}_m^\delta (s) ds
	 \end{align}
Applying 	 Gronwall's lemma, we obtain
\begin{align}
	{\bf S}_m^\delta (t)  \le \left[  B_{0}(\delta)+2TC_{1}^{\prime } \right] \exp\Big( \frac{t}{T} \ln\lf(\frac{1}{\beta_{\bf N(\delta)}  }\rt)  \Big) \le \left[  B_{0}(\delta)+2TC_{1}^{\prime } \right] \exp\Big(  \ln\lf(\frac{1}{\beta_{\bf N(\delta)}  }\rt)  \Big)= B_1(\delta,T), \label{eq2}
\end{align}
	for all $m\in
	\mathbb{N}
	,$ for all $t,$\ $0\leq t\leq T_{m}\leq T,$ \ i.e., $T_{m}=T,$ where $C_{T}$
	always indicates a bound depending on $T.$

	b) \textit{\bf The second estimate}.  Multiplying the $i^{th}$\ equation of (\ref%
	{pt1})\ by $t^2 \frac{d}{dt} c^\delta_{mi}(t)$ and summing up with respect
	to $i,$ we have
	
	\begin{align}
	\left\Vert t \frac{d}{dt}{\bf v}^\delta_{{\bf N(\delta)},m}(t) \right\Vert_{L^2(\Omega)}^2 &+t^2 \int_\Omega {\bf b}_\delta^{\text{obs}}(x,t) \nabla {\bf v}^\delta_{{\bf N(\delta)},m}(t)  \nabla \left( \frac{d}{dt} \nabla {\bf v}^\delta_{{\bf N(\delta)},m}(t)\right)  dx\nn\\
	&+ \int_\Omega t^2 F\left( {\bf v}^\delta_{{\bf N(\delta)},m}(t)\right)   \frac{d}{dt}  {\bf v}^\delta_{{\bf N(\delta)},m}(t)  dx \nn\\
	&=\int_\Omega t^2 	\mathbb{\bf P}_{\beta_{\bf N(\delta)} }^\delta \left(  {\bf v}^\delta_{{\bf N(\delta)},m}   (t)  \right)   \frac{d}{dt}  {\bf v}^\delta_{{\bf N(\delta)},m}(t)  dx. \label{pt9}
	\end{align}

	It is easy to check that for any $u \in H^1(\Omega)$
	\begin{equation}
	\frac{d}{dt} \left[ \int_\Omega {\bf b}_\delta^{\text{obs}}(x,t) |\nabla u(t)|^2 dx  \right]= 2\int_\Omega  {\bf b}_\delta^{\text{obs}}(x,t)  \nabla u(t) \nabla u'(t) dx+  \int_\Omega \frac{\partial }{\partial t}{\bf b}_\delta^{\text{obs}}(x,t) |\nabla u(t)|^2 dx.
	\end{equation}
The equality \eqref{pt9} is equivalent to
\begin{align}
2	\left\Vert t \frac{d}{dt}{\bf v}^\delta_{{\bf N(\delta)},m}(t) \right\Vert_{L^2(\Omega)}^2&+ 	\frac{d}{dt} \left[ t^2 \int_\Omega {\bf b}_\delta^{\text{obs}}(x,t) | {\bf v}^\delta_{{\bf N(\delta)},m}(t) |^2 dx  \right]+2\int_\Omega t^2 F\left( {\bf v}^\delta_{{\bf N(\delta)},m}(t)\right)   \frac{d}{dt}  {\bf v}^\delta_{{\bf N(\delta)},m}(t)  dx \nn\\
&= 2t \int_\Omega {\bf b}_\delta^{\text{obs}}(x,t)|\nabla {\bf v}^\delta_{{\bf N(\delta)},m}(t) |^2  dx+ t^2 \int_\Omega \frac{\partial}{\partial t} {\bf b}_\delta^{\text{obs}}(x,t)|\nabla {\bf v}^\delta_{{\bf N(\delta)},m}(s) |^2  dx\nn\\
&+\int_\Omega t^2 	\mathbb{\bf P}_{\beta_{\bf N(\delta)} }^\delta \left(  {\bf v}^\delta_{{\bf N(\delta)},m}   (t)  \right)   \frac{d}{dt}  {\bf v}^\delta_{{\bf N(\delta)},m}(t)  dx.
\end{align}
By integrating the last equality from $0$ to $t$, we get
\begin{align}
&2 \int_0^t \left\Vert s \frac{d}{ds}{\bf v}^\delta_{{\bf N(\delta)},m}(s) \right\Vert_{L^2(\Omega)}^2ds+ \underbrace {t^2 \int_\Omega {\bf b}_\delta^{\text{obs}}(x,t) | {\bf v}^\delta_{{\bf N(\delta)},m}(t) |^2 dx}_{I_1}\nn\\
& + \underbrace{_2 \int_0^t \int_\Omega s^2 F\left( {\bf v}^\delta_{{\bf N(\delta)},m}(s)\right)   \frac{d}{ds}  {\bf v}^\delta_{{\bf N(\delta)},m}(s)  dx ds}_{I_2}\nn\\
&= \underbrace {2 \int_0^t  \int_\Omega s  {\bf b}_\delta^{\text{obs}}(x,s)|\nabla {\bf v}^\delta_{{\bf N(\delta)},m}(s) |^2  dxds}_{I_3} + \underbrace {\int_0^t   \int_\Omega s^2  \frac{\partial}{\partial s} {\bf b}_\delta^{\text{obs}}(x,s)|\nabla {\bf v}^\delta_{{\bf N(\delta)},m}(s) |^2  dxds}_{I_4}\nn\\
&+\underbrace {\int_0^t \int_\Omega   s^2	\mathbb{\bf P}_{\beta_{\bf N(\delta)} }^\delta \left(  {\bf v}^\delta_{{\bf N(\delta)},m}   (s)  \right)   \frac{d}{ds}  {\bf v}^\delta_{{\bf N(\delta)},m}(s)  dxds}_{I_5}.
\end{align}
{\bf Estimate $I_1$}.
Since the assumption $ {\bf b}_\delta^{\text{obs}}(x,t)  \ge b_0$, we know that
\begin{align}
I_1 = t^2 \int_\Omega {\bf b}_\delta^{\text{obs}}(x,t) | {\bf v}^\delta_{{\bf N(\delta)},m}(t) |^2 dx \geq b_{0}\left\Vert t {\bf v}^\delta_{{\bf N(\delta)},m}(t) \right\Vert
_{H^{1}}^{2}. \label{pt10}
\end{align}
{\bf Estimate $I_2$}. To estimate $I_2$, we need the following Lemma
\begin{lemma} \label{lem1}
	Let $\mu_0=\left( \frac{C'_1}{C_1} \right)^{1/p} ,~ \overline m= \int_{-\mu_0}^{+\mu_0} |F (\xi)|d\xi,~~\widetilde F(z)= \int_0^z F(y)dy,~~z \in \mathbb{R}.$ Then we get
	\begin{equation}
	-\overline m \le \widetilde F(z) \le C_2 \left(  |z| +\frac{1}{p} |z|^p \right),~~z \in \mathbb{R}.
	\end{equation}
\end{lemma}
The proof of Lemma \eqref{lem1} is easy and we omit it here. Now we return to estimate $I_2$. By a simple computation and then using Lemma \eqref{lem1}, we have
\begin{align}
I_2 &=2\int\nolimits_{0}^{t}s^{2}ds\frac{d}{ds} \left[ \int\nolimits_\Omega  dx\int%
\nolimits_{0}^{ {\bf v}^\delta_{{\bf N(\delta)},m} (x,s)}F(y)dy\right]\nn\\
&=2\int\nolimits_{0}^{t}s^{2}ds\frac{d}{ds} \left[ \int\nolimits_\Omega  dx\int%
\nolimits_{0}^{ {\bf v}^\delta_{\bf N(\delta),m} (x,s)}F(y)dy\right] \nn\\
&=2\int\nolimits_{0}^{t}\left[ \frac{d}{ds}\left( s^{2}\int\nolimits_\Omega  \widetilde {F}%
\left( {\bf v}^\delta_{\bf N(\delta),m}(x,s) \right)  dx\right) -2s\int\nolimits_\Omega  \widetilde {F}%
\left( {\bf v}^\delta_{\bf N(\delta),m}(x,s) \right)  dx\right] \nn\\
&=2t^{2}\int\nolimits_\Omega  \widetilde {F}%
\left( {\bf v}^\delta_{\bf N(\delta),m}(x,t) \right)  dx-4\int%
\nolimits_{0}^{t}sds \int\nolimits_\Omega  \widetilde {F}%
\left( {\bf v}^\delta_{\bf N(\delta),m}(x,s) \right)  dx  \nn\\
&\geq
-2T^{2} \overline m |\Omega|-4C_{2}\int\nolimits_{0}^{t}s\left[ \left\Vert
{\bf v}^\delta_{\bf N(\delta),m} (s)\right\Vert _{L^{1}}+\frac{1}{p}\left\Vert {\bf v}^\delta_{\bf N(\delta),m} (s)\right\Vert
_{L^{p}}^{p}\text{ }\right] ds\nn\\
&\geq
-2T^{2} \overline m  |\Omega| -4TC_{2}\left[ T\left\Vert {\bf v}^\delta_{\bf N(\delta),m} \right\Vert _{L^{\infty
	}(0,T;L^{2})}+\frac{1}{p}\frac{1}{2C_{1}} 	{\bf S}_m^\delta (t)  \right]\nn\\
& \geq - B_2(\delta, T). \label{pt11}
\end{align}
{\bf Estimate $I_3$}. Using \eqref{eq1}, we have the following estimate
\begin{align}
I_3  \le 2T  b_1   \int_0^t \|{\bf v}^\delta_{\bf N(\delta),m} (s) \|^2_{H^1}ds \le \frac{2Tb_1}{2b_0} 	{\bf S}_m^\delta (t). \label{pt12}
\end{align}	

{\bf Estimate $I_4$}. Let us set
$$\widetilde{a}
_{T}= \underset{%
	(x,t)\in \lbrack 0,1]\times \lbrack 0,T]}{\sup }   \frac{\partial}{\partial t} {\bf b}_\delta^{\text{obs}}(x,t) $$ then $I_4$ is bounded by
\begin{align}
I_4   &\leq \widetilde{a}_{T}\int\nolimits_{0}^{t}\left
\Vert s  {\bf v}^\delta_{\bf N(\delta),m} (s) \right\Vert _{H^{1}}^{2}ds \le  T^{2}\widetilde{a}
_{T}\int\nolimits_{0}^{t}\left\Vert  {\bf v}^\delta_{\bf N(\delta),m} (s) \right\Vert
_{H^{1}}^{2}ds \leq \frac{T^{2}\widetilde{a}_{T}}{a_{0}} 	{\bf S}_m^\delta (t). \label{pt13}
\end{align}	
{\bf Estimate $I_5$}.
Using Lemma \eqref{importantlemma}, we obtain the following estimate for $I_5$
	\begin{align}
	I_5 &\le 2 \int_0^t  \|  	s  \mathbb{\bf P}_{\beta_{\bf N(\delta)} }^\delta  ({\bf v}^\delta_{\bf N(\delta),m} (s)) \| \|s \frac{d}{ds}{\bf v}^\delta_{\bf N(\delta),m}(s)   \| ds \nn\\
	& \le \int_0^t  \|  	s  \mathbb{\bf P}_{\beta_{\bf N(\delta)} }^\delta  ({\bf v}^\delta_{\bf N(\delta),m} (s)) \|^2ds+ \int_0^t  \|s \frac{d}{ds}{\bf v}^\delta_{\bf N(\delta),m}(s)   \|^2 ds \nn\\
	&\le  \ln^2\lf(\frac{1}{\beta_{\bf N(\delta)} }\rt) \int_0^t  \| {\bf v}^\delta_{\bf N(\delta),m}(s)   \|^2 ds+ \int_0^t  \|s \frac{d}{ds}{\bf v}^\delta_{\bf N(\delta),m}(s)   \|^2 ds\nn\\
	&\le \ln^2\lf(\frac{1}{\beta_{\bf N(\delta)} }\rt)  \frac{	{\bf S}_m^\delta (t) }{a_0}+ \int_0^t  \|s \frac{d}{ds}{\bf v}^\delta_{\bf N(\delta),m}(s)   \|^2 ds
	\end{align}
Combining \eqref{pt10}, \eqref{pt11}, \eqref{pt12}, \eqref{pt13},we obtain
\begin{align}
2 \int_0^t \left\Vert s \frac{d}{ds}{\bf v}^\delta_{{\bf N(\delta)},m}(s) \right\Vert_{L^2(\Omega)}^2ds &+ b_{0}\left\Vert t {\bf v}^\delta_{{\bf N(\delta)},m}(t) \right\Vert
_{H^{1}}^{2} \nn\\
 &\le B_2(\delta,T)+  \frac{2Tb_1}{2b_0} 	{\bf S}_m^\delta (t)+\frac{T^{2}\widetilde{a}_{T}}{a_{0}} 	{\bf S}_m^\delta (t) \nn\\
&+\ln^2\lf(\frac{1}{\beta_{\bf N(\delta)} }\rt)  \frac{	{\bf S}_m^\delta (t)}{a_0}+ \int_0^t  \|s \frac{d}{ds}{\bf v}^\delta_{\bf N(\delta),m}(s)   \|^2 ds. \label{pt14}
\end{align}
 Let us set
 \[
 {\bf R}_m^\delta (t)= \int_0^t \left\Vert s \frac{d}{ds}{\bf v}^\delta_{{\bf N(\delta)},m}(s) \right\Vert_{L^2(\Omega)}^2ds +  \left\Vert t {\bf v}^\delta_{{\bf N(\delta)},m}(t) \right\Vert
 _{H^{1}(\Omega) }^{2}.
 \]
then since $$\int_0^t  {\bf R}_m^\delta (s)ds \ge  \int_0^t  \|s \frac{d}{ds}{\bf v}^\delta_{\bf N(\delta),m}(s)   \|^2 ds $$ together with \eqref{pt14}, we deduce that
\begin{align}
 {\bf R}_m^\delta (t) \le \frac{ B(3,\delta) } {\min (2, b_0)}+\frac{1}{\min (2, b_0)}  \int_0^t  {\bf R}_m^\delta (s)ds.
\end{align}
where
$$B(3,\delta) = B_2(\delta,T)+  \frac{2Tb_1}{2b_0} B(2,\delta) +\frac{T^{2}\widetilde{a}_{T}}{a_{0}} 	B(2,\delta) +\ln^2\lf(\frac{1}{\beta_{\bf N(\delta)} }\rt)  \frac{B(2,\delta)}{a_0}.
$$
 Applying Gronwall's inequlality, we obtain that
 \begin{equation}
 \int_0^t \left\Vert s \frac{d}{ds}{\bf v}^\delta_{{\bf N(\delta)},m}(s) \right\Vert_{L^2(\Omega)}^2ds +  \left\Vert t {\bf v}^\delta_{{\bf N(\delta)},m}(t) \right\Vert
 _{H^{1}(\Omega) }^{2} \le B_4(\delta,T), \label{eq3}
 \end{equation}
 where $B(4,\delta) $ depends only on $\delta, T$ and does not depend on $m$. \\
{\bf Step 3}. {\it The limiting process}.
 	
 	Combining \eqref{eq1}, \eqref{eq2} and \eqref{eq3},  we deduce that, there exists a
 	subsequence of $ \{ {\bf v}^\delta_{{\bf N(\delta)},m}  \}$ still denoted by $ \{ {\bf v}^\delta_{{\bf N(\delta)},m}  \}$  such that (see \cite{Lions})
 	\begin{equation} \label{eq7}
 	\left\{
 	\begin{tabular}{lll}
 	${\bf v}^\delta_{{\bf N(\delta)},m} \rightarrow {\bf v}^\delta_{{\bf N(\delta)}}$ & $\text{in}$ & $L^{\infty }(0,T;L^{2})\text{ \ weak*,}$%
 	\medskip \\
 	$ {\bf v}^\delta_{{\bf N(\delta)},m} \rightarrow {\bf v}^\delta_{{\bf N(\delta)}} $ & $\text{in}$ & $L^{2}(0,T;H^{1})\text{ \ weak,}$%
 	\medskip \\
 	$t {\bf v}^\delta_{{\bf N(\delta)},m} \rightarrow t{\bf v}^\delta_{{\bf N(\delta)}}  $ & $\text{in}$ & $L^{\infty }(0,T;H^{1})\text{ \ weak*,%
 	}$\medskip \\
 	$\left(t {\bf v}^\delta_{{\bf N(\delta)},m} \right)^{\prime }\rightarrow   \left(t {\bf v}^\delta_{{\bf N(\delta)}} \right)^{\prime } $ & $\text{in}$ & $L^{2}(Q_{T})%
 	\text{ \ weak,}$\medskip \\
 	${\bf v}^\delta_{{\bf N(\delta)},m} \rightarrow {\bf v}^\delta_{{\bf N(\delta)}}$ & $\text{in}$ & $L^{p}(Q_{T})\text{ \ weak.}$%
 	\end{tabular}%
 	\right.
 	\end{equation}

 	Using a compactness lemma (\cite{Lions}, Lions, p. 57) applied to $\eqref{eq7}$, we can extract from the sequence  $ \{ {\bf v}^\delta_{{\bf N(\delta)},m}  \}$  a subsequence still
 	denoted by  $ \{ {\bf v}^\delta_{{\bf N(\delta)},m}  \}$ such that%
 	\begin{equation}
 	\begin{tabular}{l}
 	$\left(t {\bf v}^\delta_{{\bf N(\delta)},m} \right)^{\prime }\rightarrow   \left(t {\bf v}^\delta_{{\bf N(\delta)}} \right)^{\prime }  \text{\ \ strongly in }L^{2}(Q_{T}).$%
 	\end{tabular}
 	\end{equation}
 	By the Riesz-Fischer theorem, we can extract from  $ \{ {\bf v}^\delta_{{\bf N(\delta)},m}  \}$ a subsequence
 	still denoted by  $ \{ {\bf v}^\delta_{{\bf N(\delta)},m}  \}$  such that%
 	\begin{equation}
 	\begin{tabular}{l}
 	${\bf v}^\delta_{{\bf N(\delta)},m}(x,t) \rightarrow {\bf v}^\delta_{{\bf N(\delta)}}(x,t)\text{ \ a.e. \ }(x,t)\ \ \text{in \ }%
 	Q_{T}=\Omega \times (0,T).$%
 	\end{tabular}
 	\tag{3.40}  \label{c40}
 	\end{equation}
 	Because $F$ is continuous, then%
 	\begin{equation}
 	F \left( x,t, {\bf v}^\delta_{{\bf N(\delta)},m}(x,t) \right)\rightarrow F \left( x,t, {\bf v}^\delta_{{\bf N(\delta)}}(x,t) \right) \text{ \ a.e. \ }(x,t)\ \ \text{in \ }%
 	Q_{T}=\Omega  \times (0,T).  \label{eq4}
 	\end{equation}
 	On the other hand, using \eqref{ass2},  \eqref{eq1}, \eqref{eq2} , we obtain

 	\begin{equation}
 	\left\Vert  	F \left(  {\bf v}^\delta_{{\bf N(\delta)},m}(x,t) \right)   \right \Vert_{L^{p^{\prime }}(Q_{T})}  \leq B_{5} (\delta,T) , \label{eq5}
 	\end{equation}
 	where $B_{5} (\delta,T)  $ is a constant independent of $m.$
 	We shall now require the following lemma, the proof of which can be found in
 	\cite{Lions} (see Lemma 1.3).

 	\begin{lemma} \label{lemmalions}
 		Let $Q$ be a bounded open subset of $\mathbb{R}
 		^{N}$ and  $G_{m},$ $G\in L^{q}(Q),$ $1<q<\infty ,$ such that
 		\begin{equation}
 		\left\Vert G_{m}\right\Vert _{L^{q}(Q)}\leq C,\text{ \textit{where} }C\text{
 			\textit{is a constant independent of}\ }m
 		\end{equation}
 		and
 			\begin{equation*}
 		G_{m}\rightarrow G\text{ \ a.e. \ }(x,t)\ \ \text{in \ }Q.
 			\end{equation*}
 			Then
 			\begin{equation*}
 			G_{m}\rightarrow G  ~\textit{in }~~~~ L^{q}(Q)\textit{weakly.}
 			\end{equation*}
 	\end{lemma}
 	Applying Lemma \eqref{lemmalions}  with  $q=p^{\prime }=\frac{p}{p-1}, \ G_{m}=F \left(  {\bf v}^\delta_{{\bf N(\delta)},m}(x,t) \right)  ,$ $G=F \left(  {\bf v}^\delta_{{\bf N(\delta)}}(x,t) \right)  ,$
 we deduce from  \eqref{eq4} and \eqref{eq5} that
 	
 	\begin{equation}   \label{eq8}
 	F \left(  {\bf v}^\delta_{{\bf N(\delta)},m} \right) \rightarrow F \left(  {\bf v}^\delta_{{\bf N(\delta)}} \right)  ~\textit{in }~~~~ L^{p'}(Q)~~~~\textit{weakly.}
 	\end{equation}
 	Passing to the limit in \eqref{pt1} and \eqref{pt00}  by  \eqref{eq7} and \eqref{eq8},
 	we have established equation \eqref{pt11}.

\subsubsection{Proof of the uniqueness of  solution of Problem \eqref{QR3} }
Assume that the Problem \eqref{QR3} has two solution
${\bf v}^\delta_{\bf N(\delta)}$ and ${\bf w}^\delta_{\bf N(\delta)}$. We have to show that ${\bf v}^\delta_{\bf N(\delta)} = {\bf w}^\delta_{\bf N(\delta)}$.
 We recall that
\bq
\left\{ \begin{gathered}
\frac{\partial  {\bf v}^\delta_{\bf N(\delta)} }{\partial t} + \nabla \Big(	{\bf b}_\delta^{\text{obs}}(x,t)\nabla   {\bf v}^\delta_{\bf N(\delta)}  \Big)= F \lf(x,t,{\bf v}^\delta_{\bf N(\delta)} (x,t)\right)
+ \mathbb{\bf P}_{\beta_{\bf N(\delta)} }^\delta  {\bf v}^\delta_{\bf N(\delta)}, \hfill \\
\frac{\partial  {\bf w}^\delta_{\bf N(\delta)} }{\partial t} + \nabla \Big(	{\bf b}_\delta^{\text{obs}}(x,t)\nabla   {\bf w}^\delta_{\bf N(\delta)}  \Big)= F \lf(x,t,{\bf w}^\delta_{\bf N(\delta)} (x,t)\right)
+ \mathbb{\bf P}_{\beta_{\bf N(\delta)} }^\delta  {\bf w}^\delta_{\bf N(\delta)} \hfill\\
{\bf u}^\delta_{\bf N(\delta)} (x,T)={\bf w}^\delta_{\bf N(\delta)}=\overline  G_{\delta, {\bf N(\delta)}}(x),\hfill  \label{system1}
\end{gathered}  \right.
\eq
	For $\overline R_\delta>0$, we put $$ \mathbf{W}^\delta_{\bf N(\delta)} (x,t)=e^{\overline R_\delta (t-T)}\lf[ {\bf v}^\delta_{\bf N(\delta)}(x,t)- {\bf w}^\delta_{\bf N(\delta)}(x,t)\rt].$$ Then for $ (x,t) \in \Omega\times (0,T)$, we get
		\begin{align} \label{eq13}
		\frac{\partial \mathbf{W}^\delta_{\bf N(\delta)}}{\partial t} &+ \nabla \Big({\bf b}_\delta^{\text{obs}}(x,t)\nabla  \mathbf{W}^\delta_{\bf N(\delta)}\Big)-\overline R_\delta \mathbf{W}^\delta_{\bf N(\delta)} \nn\\
		&=
		\mathbb{\bf P}_{\beta_{\bf N(\delta)} }^\delta   \mathbf{W}^\delta_{\bf N(\delta)} + e^{\overline R_\delta(t-T)}\lf[F \lf(x,t,{\bf v}^\delta_{\bf N(\delta)} (x,t)\right)-F \lf(x,t,{\bf w}^\delta_{\bf N(\delta)} (x,t)\right) \rt],
		\end{align}
		and $$\mathbf{W}^\delta_{\bf N(\delta)}|_{\partial \Omega}=0, ~ \mathbf{W}^\delta_{\bf N(\delta)}(x,T)=0 .$$
	By taking the inner product of the two sides of \eqref{eq13} with $\mathbf{W}^\delta_{\bf N(\delta)}$ then taking the integral from $t$ to $T$ and noting the  equality $$\int_\Omega \nabla \Big(	{\bf b}_\delta^{\text{obs}}(x,t)\nabla  \mathbf{W}^\delta_{\bf N(\delta)} \Big) \mathbf{V}^\delta_{\bf N(\delta)} dx=-\int_\Omega {\bf b}_\delta^{\text{obs}} (x,t) |\nabla  \mathbf{W}^\delta_{\bf N(\delta)} |^2 dx,$$ we deduce
\begin{align}
&\|\mathbf{W}^\delta_{\bf N(\delta)} (.,T)\|^2_{L^2(\Omega)} - \|\mathbf{W}^\delta_{\bf N(\delta)} (.,t)\|^2_{L^2(\Omega)} \nn\\
&= 2 \int_t^T \int_\Omega 	\mathbb{\bf P}_{\beta_{\bf N(\delta)} }^\delta   \mathbf{W}^\delta_{\bf N(\delta)}(x,s)dxds+2 \int_t^T \int_{\Omega}  {\bf b}_\delta^{\text{obs}} (x,s) |\nabla  \mathbf{W}^\delta_{\bf N(\delta)} |^2 dxds + 2 \overline R_\delta \int_t^T \|\mathbf{W}^\delta_{\bf N(\delta)} (.,s)\|^2_{L^2(\Omega)}ds \nn\\
& +2 \int_t^T \int_\Omega e^{\overline R_\delta(t-T)}\lf[F \lf(x,s,{\bf v}^\delta_{\bf N(\delta)} (x,s)\right)-F \lf(x,s,{\bf w}^\delta_{\bf N(\delta)} (x,s)\right) \rt] \mathbf{W}^\delta_{\bf N(\delta)}(x,s)dxds\nn\\
&\ge 2 \int_t^T \int_\Omega 	\mathbb{\bf P}_{\beta_{\bf N(\delta)} }^\delta   \mathbf{W}^\delta_{\bf N(\delta)}(x,s)dxds+ 2 \overline R_\delta \int_t^T \|\mathbf{W}^\delta_{\bf N(\delta)} (.,s)\|^2_{L^2(\Omega)}ds\nn\\
& +2 \int_t^T \int_\Omega e^{\overline R_\delta(t-T)}\lf[F \lf(x,s,{\bf v}^\delta_{\bf N(\delta)} (x,s)\right)-F \lf(x,s,{\bf w}^\delta_{\bf N(\delta)} (x,s)\right) \rt] \mathbf{W}^\delta_{\bf N(\delta)}(x,s)dxds \label{eq14}
\end{align}
By  the   assumption we have
\begin{align}
&\int_t^T \int_\Omega e^{\overline R_\delta(s-T)}\lf[F \lf(x,s,{\bf v}^\delta_{\bf N(\delta)} (x,s)\right)-F \lf(x,s,{\bf w}^\delta_{\bf N(\delta)} (x,s)\right) \rt] \mathbf{W}^\delta_{\bf N(\delta)}(x,s)dxds\nn\\
&= \int_t^T \int_\Omega e^{\overline R_\delta(s-T)}\lf[F \lf(x,s,{\bf v}^\delta_{\bf N(\delta)} (x,s)\right)-F \lf(x,s,{\bf w}^\delta_{\bf N(\delta)} (x,s)\right) \rt] e^{\overline R_\delta (s-T)}\lf[ {\bf v}^\delta_{\bf N(\delta)}(x,s)- {\bf w}^\delta_{\bf N(\delta)}(x,s)\rt] dxds \nn\\
&\ge -\overline \gamma \int_t^T \int_\Omega e^{2\overline R_\delta(s-T)} \lf[ {\bf v}^\delta_{\bf N(\delta)}(x,s)- {\bf w}^\delta_{\bf N(\delta)}(x,s)\rt]^2 dxds\nn\\
&= -\overline \gamma \int_t^T \|\mathbf{W}^\delta_{\bf N(\delta)} (.,s)\|^2_{L^2(\Omega)}ds. \label{eq15}
\end{align}
Using the  inequality \eqref{bound P},  we get the following estimate
\begin{equation}
 \int_t^T \int_\Omega 	\mathbb{\bf P}_{\beta_{\bf N(\delta)} }^\delta   \mathbf{W}^\delta_{\bf N(\delta)}(x,s)dxds \ge -  \frac{2}{T} \ln\lf(\frac{1}{\beta_{\bf N_\delta}}\rt) \int_t^T \|\mathbf{W}^\delta_{\bf N(\delta)} (.,s)\|^2_{L^2(\Omega)}ds. \label{eq16}
\end{equation}
Combine equations   \eqref{eq14}, \eqref{eq15}, \eqref{eq16}   and choose
\[
\overline R_\delta=  \frac{1}{T} \ln\lf(\frac{1}{\beta_{\bf N_\delta}}\rt)+\gamma
\]	
	to obtain
	$$
	\|\mathbf{W}^\delta_{\bf N(\delta)} (.,T)\|^2_{L^2(\Omega)} - \|\mathbf{W}^\delta_{\bf N(\delta)} (.,t)\|^2_{L^2(\Omega)}  \ge 0
	$$
	This implies that for all $t \in [0,T]$ then $\|\mathbf{W}^\delta_{\bf N(\delta)} (.,t)\|^2_{L^2(\Omega)} =0$ since $\mathbf{W}^\delta_{\bf N(\delta)} (x,T)=0$. The proof is completed.
	
	\subsubsection{Convergence estimate}
	Our analysis and proof is short and similar to the   proof of Theorem \eqref{nonlocal1}.
	Indeed, let us also set 	 $$ \mathbf{V}^\delta_{\bf N(\delta)} (x,t)=e^{\rho_\delta (t-T)}\lf[ {\bf u}^\delta_{\bf N(\delta)}(x,t)-\textbf{u}(x,t)\rt].$$
	By using  some of steps as above, we obtain
	\begin{align}
	&\|\mathbf{V}^\delta_{\bf N(\delta)}(\cdot, T)\|^2_{L^2(\Omega)} - \|\mathbf{V}^\delta_{\bf N(\delta)}(\cdot, t)\|^2_{L^2(\Omega)} \nn\\
	&~~~~~~~~~= \widetilde A_4+ \widetilde A_5+ \widetilde A_6+ \underbrace{2\int_t^T \lf\langle e^{\rho_\delta(s-T)}\lf[F \lf(x,s,{\bf u}^\delta_{\bf N(\delta)} (x,s)\right)- F\big(x,s;\textbf{u}(x,s)\big)\rt], \mathbf{V}^\delta_{\bf N(\delta)} \rt\rangle_{L^2(\Omega)}ds}_{=:\widetilde{A_8}}
	\end{align}
	The terms $\widetilde A_4, \widetilde A_5, \widetilde A_6$ is similar to \eqref{3J}.  Now, we consider $\widetilde A_8$.
	By    assumption \eqref{ass3},  we have
	\begin{align}
	&\int_t^T \int_\Omega e^{\overline R_\delta(s-T)}\lf[F \lf(x,s,{\bf u}^\delta_{\bf N(\delta)} (x,s)\right)-F \lf(x,s,{\bf u} (x,s)\right) \rt] \mathbf{V}^\delta_{\bf N(\delta)}(x,s)dxds\nn\\
	&= \int_t^T \int_\Omega e^{\overline R_\delta(s-T)} \lf[F \lf(x,s,{\bf u}^\delta_{\bf N(\delta)} (x,s)\right)-F \lf(x,s,{\bf u} (x,s)\right) \rt]  e^{\overline R_\delta (s-T)}\lf[ {\bf u}^\delta_{\bf N(\delta)}(x,s)- {\bf u }(x,s)\rt] dxds \nn\\
	&\ge -\overline \gamma \int_t^T \int_\Omega e^{2\overline R_\delta(s-T)} \lf[ {\bf u}^\delta_{\bf N(\delta)}(x,s)- {\bf u }(x,s)\rt]^2  dxds\nn\\
	&= -\overline \gamma \int_t^T \|\mathbf{V}^\delta_{\bf N(\delta)} (.,s)\|^2_{L^2(\Omega)}ds. \label{eq15}
	\end{align}
		After using the results of  the proof of Theorem \ref{nonlocal1}, we get
		\begin{align}
		{\bf E} \|\mathbf{V}^\delta_{\bf N(\delta)}(\cdot, t)\|^2_{L^2(\Omega)} & \leq {\bf E} \| \overline  G_{\delta, {\bf N(\delta)}}(x)-g(x) \|^2_{L^2(\Omega)}\nn\\
		&+  \beta_{\bf N_\delta}  \|\textbf{u}\|_{C\lf([0,T];\mathcal W_{MT}(\Omega)\rt)}^2 +  \frac{\delta^2 T^3 }{b_0} \lf\| \textbf{u}  \rt\|_{L^\infty(0,T;\mathcal H_0^1(\Omega))}^2 \nn\\
		&\quad + 	{\bf E}  \int_t^T \lf(-2\rho_\delta+ \frac{2}{T} \ln\lf(\frac{1}{\beta_{\bf N_\delta}}\rt) + 2 \overline \gamma+1 \rt)\|\mathbf{V}^\delta_{\bf N(\delta)} (\cdot,s) \|^2_{L^2(\Omega)} ds.
		\end{align}
		Since $$\mathbf{V}^\delta_{\bf N(\delta)}(x,t)=e^{\rho_\delta(t-T)} \Big( {\bf u}^\delta_{\bf N(\delta)}(x,t)-\textbf{u}(x,t)\Big) $$  and applying Lemma \ref{lemmawhitenoise},  we observe that
		\begin{align}
		&e^{2\rho_\delta(t-T)} {\bf E} \lf\|{\bf u}^\delta_{\bf N(\delta)}(x,t)-\textbf{u}(x,t) \rt\|^2_{L^2(\Omega)}
		\leq \delta^2 {\bf N}(\delta)  + \frac{1}{\la_{\bf N(\delta)}^{2\gamma}} \|g\|_{H^{2\gamma}(\Omega)}   \nn\\
		&~~~~~~~~~~~~~~~~~~~~~~~~~~~~~~~~+ \beta_{\bf N_\delta}  \|\textbf{u}\|_{C\lf([0,T];\mathcal W_{MT}(\Omega)\rt)}^2 +  \frac{\delta^2 T^3 }{b_0} \lf\| \textbf{u}  \rt\|_{L^\infty(0,T;\mathcal H_0^1(\Omega))}^2
		\nn\\
		&~~~~~~~~~~~~~~~~~~~~~~~~~~~~~~~~+ \lf( 2 \overline \gamma +1 \rt) \int_t^T  	e^{2\rho_\delta(s-T)} {\bf E} \lf\|{\bf u}^\delta_{\bf N(\delta)}(x,s)-\textbf{u}(x,s) \rt\|^2_{L^2(\Omega)}  ds. \label{a}
		\end{align}
		Gronwall's lemma allows us to obtain	\begin{align}  \label{mainerror}
		&e^{2\rho_\delta(t-T)} {\bf E} \lf\|{\bf u}^\delta_{\bf N(\delta)}(x,t)-\textbf{u}(x,t) \rt\|^2_{L^2(\Omega)} \nn\\
		&\leq \underbrace{\left[\delta^2 {\bf N}(\delta)  + \frac{1}{\la_{\bf N(\delta)}^{2\gamma}} \|g\|_{H^{2\gamma}(\Omega)}
		+ \beta_{\bf N_\delta}  \|\textbf{u}\|_{C\lf([0,T];\mathcal W_{MT}(\Omega)\rt)}^2 +  \frac{\delta^2 T^3 }{b_0} \lf\| \textbf{u}  \rt\|_{L^\infty(0,T;\mathcal H_0^1(\Omega))}^2  \right]}_{ \widetilde C (\delta) } e^{( 2  \overline \gamma +1 )(T-t)}.
		\end{align}
		By choosing $\rho_\delta= \frac{1}{T} \ln\lf(\frac{1}{\beta_{\bf N_\delta}}\rt) >0$ we have
		\begin{align} \label{resulproof-a-F_lip}
		{\bf E} \lf\|{\bf u}^\delta_{\bf N(\delta)}(x,t)-\textbf{u}(x,t) \rt\|^2_{L^2(\Omega)}
		\leq \beta_{\bf N(\delta)}^{\frac{2t}{T}} e^{( 2 \overline \gamma +1 )T} 	\widetilde C (\delta).
		\end{align}
		\section{Application to some specific equations}\label{section7}
		\subsection{Ginzburg-Landau equation}
		Here  we consider a special source function $F
		(u)= u-u^3$ for  Problem \eqref{parabolicproblemwhitenoise}. This is  called   Ginzburg-Landau equation. This function satisfies the condition of section \ref{section4} and  does not satisfy that the condition in section \ref{section5}.   For all $\mathcal R>0$, we approximate $F$ by $\mathcal{F}_{\mathcal R}$ defined by
		\begin{align} \label{df-F-2}
		\mathcal{F}_{\mathcal R}(x,t;w) :=
		\begin{cases}
		\mathcal R^3 - \mathcal R, & w \in (-\infty,-\mathcal R)\\
u-u^3, & w \in [-\mathcal R, \mathcal R],\\
\mathcal R- \mathcal R^3, & w \in (\mathcal R,+\infty).
		\end{cases}
		\end{align}
	We consider the problem 			
		\bq  \label{QR5}
		\left\{ \begin{gathered}
		\frac {\partial  {\bf u}^{\delta}_{\bf N(\delta)} }{\partial t}-\nabla\Big(	{\bf a}_\delta^{\text{obs}}(x,t)\nabla {\bf u}^\delta_{\bf N(\delta)}  \Big)-  \mathbb{\bf Q}_{\beta_{\bf N(\delta)} }^\delta  ({\bf u}^\delta_{\bf N(\delta)} ) (x,t)  \hfill \\
		~~~~~~~~~~~~~~~~~~~~~~~~~~~~~~~~~~~~~=\mathcal{F}_{R_\delta} \lf({\bf u}^\delta_{\bf N(\delta)}  (x,t)\rt),\quad (x,t) \in \Omega \times (0,T), \hfill \\
		{\bf u}^\delta_{\bf N(\delta)} |_{\partial \Omega }=0,\quad t \in (0,T),\hfill\\
		{\bf u}^\delta_{\bf N}  (x,T)=\overline  G_{\delta, {\bf N(\delta)}}(x), \quad (x,t) \in \Omega \times (0,T),\hfill
		\end{gathered}  \right.
		\eq
			It is easy to see that $K
			(\mathcal R_\delta)= 1+3 \mathcal R_\delta^2$.
			 Let us choose
			 $\beta_{\bf N(\delta)}= {\bf N(\delta)}^{-c}$ for any $0<c<  \min(\frac{1}{2}, \frac{2\gamma}{d})$. And ${\bf N(\delta)}$ is chosen as follows
			 \begin{equation}
			 {\bf N(\delta)}= \left(  \frac{1}{\delta} \right)^{m(\frac{1}{2}-c)},\beta_{\bf N(\delta)} = \left(  \frac{1}{\delta} \right)^{-mc(\frac{1}{2}-c)} ~~0<m<1.
			 \end{equation}
			 and
			  choose $\mathcal R_\delta$ such that
			 \[
		{\mathcal R}_\delta=\sqrt { \frac{K\left( 	{\mathcal R}_\delta \right) -1 }{3} }= 		\sqrt { \frac{	\frac{1}{kT}  \ln \Big(  m(\frac{1}{2}-c  ) \ln \left(  \frac{1}{\delta}	\right) \Big) -1 }{3} }  .
			 \]
		Then applying Theorem \eqref{nonlocal1}, the error $	{\bf E} \lf\|{\bf u}^\delta_{\bf N(\delta)}(x,t)-\textbf{u}(x,t) \rt\|^2_{L^2(\Omega)}   $ is of order
		\[  \ln^2 \left(  \frac{1}{\delta} \right)
		\Big(  \delta \Big)^{2mc(\frac{1}{2}-c)\frac{t}{T} }.
		\]	
		
		\subsection{The nonlinear
			Fisher--KPP equation}
		In this subsection, we are  concerned with  the backward problem for
		a nonlinear parabolic equation of the Fisher--Kolmogorov--Petrovsky--Piskunov type in the following
		\begin{equation}
		\mathbf u_t-\nabla\Big(a(x,t)\nabla \mathbf u\Big)= \gamma(x)  {\mathbf u}^2- \mu(x) {\mathbf u} ,\quad (x,t) \in \Omega \times (0,T), \label{Fisher1}
		\end{equation}
		with the following condition
		\bq
		\left\{ \begin{gathered}
		\mathbf u(x,T)=g(x),\quad  (x,t) \in \Omega \times (0,T),\hfill\\
		\mathbf u|_{\partial \Omega}=0,\quad  t \in (0,T),\hfill
		  \label{Fisher2}
		\end{gathered}  \right.
		\eq
		By  Skellam \cite{Skellam}, the equation \eqref{Fisher1} has many applications in  population dynamics
		and periodic environments.  In these
		references, the quantity ${\bf u} (x,t)$ generally stands for a population density, and the coefficients $a(x,t)$,~$\gamma(x),~\mu(x)$
		 respectively, correspond to the diffusion coefficient, the intrinsic
		growth rate coefficient
		and a coefficient measuring the effects of competition on the birth and death rates.
		\subsection{The second equation}
		Taking the function $F(u) = u^{\frac{1}{3}}$. It is easy to check that $F$ satisfy \eqref{ass1}, \eqref{ass2} and \eqref{ass3}. Moreover, we can show that $F$ is not locally Lipschitz function. So, we can not  regularize problem in this case by Problem \eqref{QR2}.
			We consider the problem 			
			\bq  \label{QR6}
			\left\{ \begin{gathered}
			\frac {\partial  {\bf u}^{\delta}_{\bf N(\delta)} }{\partial t}-\nabla\Big(	{\bf a}_\delta^{\text{obs}}(x,t)\nabla {\bf u}^\delta_{\bf N(\delta)}  \Big)-  \mathbb{\bf Q}_{\beta_{\bf N(\delta)} }^\delta  ({\bf u}^\delta_{\bf N(\delta)} ) (x,t)  \hfill \\
			~~~~~~~~~~~~~~~~~~~~~~~~~~~~~~~~~~~~~= \lf({\bf u}^\delta_{\bf N(\delta)}  (x,t)\rt)^{\frac{1}{3}},\quad (x,t) \in \Omega \times (0,T), \hfill \\
			{\bf u}^\delta_{\bf N(\delta)} |_{\partial \Omega }=0,\quad t \in (0,T),\hfill\\
			{\bf u}^\delta_{\bf N}  (x,T)=\overline  G_{\delta, {\bf N(\delta)}}(x), \quad (x,t) \in \Omega \times (0,T),\hfill
			\end{gathered}  \right.
			\eq
		Let us choose $\beta_{\bf N_\delta}$ and ${\bf N_\delta}$ be as in subsection 6.1.
	Applying Theorem \ref{nonlocal1}, the error between the solution of Problem \eqref{QR6} and $\bf u$, $	{\bf E} \lf\|{\bf u}^\delta_{\bf N(\delta)}(x,t)-\textbf{u}(x,t) \rt\|^2_{L^2(\Omega)}   $, is of order
		$
		  \delta ^{2mc(\frac{1}{2}-c)\frac{t}{T} }.
		$

\end{document}